\documentclass{article}

\usepackage[utf8]{inputenc}
\usepackage[english]{babel}

\usepackage{amsmath, amsthm, amssymb, amsfonts}
\usepackage{mathtools}

\usepackage{tikz-cd}

\usepackage{hyperref}
\usepackage[capitalize,noabbrev]{cleveref}

\usepackage{enumitem}

\newcommand{\sslash}{/\mkern-6mu/}

\theoremstyle{definition}
\newtheorem{defn}{Definition}[section]

\theoremstyle{plain}
\newtheorem{lem}[defn]{Lemma}

\theoremstyle{plain}
\newtheorem{cor}[defn]{Corollary}

\theoremstyle{plain}
\newtheorem{prop}[defn]{Proposition}

\theoremstyle{plain}
\newtheorem{thm}[defn]{Theorem}

\theoremstyle{plain}
\newtheorem*{thm*}{Theorem}

\theoremstyle{plain}
\newtheorem{introthm}{Theorem}

\theoremstyle{plain}
\newtheorem{exmpl}[defn]{Example}

\theoremstyle{remark}
\newtheorem{rmk}[defn]{Remark}

\theoremstyle{remark}
\newtheorem{notation}[defn]{Notation}

\crefname{lem}{Lemma}{Lemmas}
\crefname{defn}{Definition}{Definitions}
\crefname{cor}{Corollary}{Corollaries}
\crefname{prop}{Proposition}{Propositions}
\crefname{thm}{Theorem}{Theorems}
\crefname{rmk}{Remark}{Remarks}
\crefname{notation}{Notation}{Notation}
\crefname{section}{Section}{Sections}

\newcommand{\Z}{{\mathbb{Z}}}
\newcommand{\N}{{\mathbb{N}}}
\newcommand{\Q}{{\mathbb{Q}}}

\newcommand{\E}{{\mathbb{E}}}

\renewcommand{\S}{{\mathbb S}}

\newcommand{\PrShv}[1]{{\mathcal{P}(#1)}}

\newcommand{\colim}[1]{{\underset{\substack{#1}}{\operatorname{colim}}\,}}
\newcommand{\colimil}[1]{{{\operatorname{colim}}_{#1}\,}}
\renewcommand{\lim}[1]{{\underset{\substack{#1}}{\operatorname{lim}}\,}}
\newcommand{\limil}[1]{{{\operatorname{lim}}_{#1}\,}}

\newcommand{\Map}[3]{{\operatorname{Map}_{#1}({#2},{#3})}}
\newcommand{\Fun}[1]{{\operatorname{Fun}_{#1}}}

\newcommand{\Sp}{{\operatorname{Sp}}}

\newcommand{\op}[1]{{#1}^{\operatorname{op}}}

\newcommand{\pSus}{{\Sigma^\infty_+}}
\newcommand{\Loop}{{\Omega^\infty}}
\newcommand{\pLoop}{{\Omega_*^\infty}}

\newcommand{\Fib}[1]{{\operatorname{fib}\mathopen{}\left(#1\right)\mathclose{}}}
\newcommand{\Cofib}[1]{{\operatorname{cofib}\mathopen{}\left(#1\right)\mathclose{}}}

\newcommand{\topos}[1]{{\mathcal {#1}}}
\newcommand{\spectra}[1]{{\operatorname{Sp}(\topos {#1})}}

\newcommand{\An}{{\mathcal An}}

\newcommand{\Ab}{{\mathcal{A}b}}
\newcommand{\AbObj}[1]{{\mathcal{A}b(#1)}}
\newcommand{\Grp}[1]{{\mathcal{G}rp(#1)}}
\newcommand{\Disc}[1]{{\operatorname{Disc}(#1)}}

\newcommand{\Cat}[1]{\mathcal {#1}}

\newcommand{\tcon}[2][0]{{#2}_{\ge #1}}
\newcommand{\tcocon}[2][0]{{#2}_{\le #1}}
\newcommand{\tstruct}[1]{(\tcon{#1}, \tcocon{#1})}

\newcommand{\heart}[1]{#1^\heartsuit}

\newcommand{\finfld}[1]{\mathbb{F}_{#1}}

\newcommand{\set}[2]{\left\{\, {#1} \,\middle\vert\, {#2} \,\right\}}
\newcommand{\col}[2]{\left\{\, {#1} \,\right\}_{#2}}

\newcommand{\ShvTop}[2]{\operatorname{Shv}_{#1}({#2})}

\newcommand{\smooth}[1]{\operatorname{Sm}_{#1}} 

\newcommand{\nis}{\operatorname{nis}}
\newcommand{\zar}{\operatorname{zar}}

\newcommand{\id}[1]{\operatorname{id}_{#1}}

\author{Klaus Mattis\footnote{\href{mailto:klaus.mattis@uni-mainz.de}{klaus.mattis@uni-mainz.de}}}
\date{\today}

\title{Unstable arithmetic fracture squares in $\infty$-topoi}

\begin{document}
\maketitle

\begin{abstract}
    We show that for a large class of $\infty$-topoi 
    there exist unstable arithmetic fracture squares,
    i.e.\ squares which recover a nilpotent sheaf $F$
    as the pullback of the rationalization of $F$ with 
    the product of the $p$-completions of $F$ ranging over 
    all primes $p\in\Z$.
\end{abstract}

\tableofcontents
\newpage

\section{Introduction}

Let $R \subseteq \Q$ be a subring of the rational numbers,
and write $P$ for the set of primes $p \in \Z$ that are invertible in $R$.
Note that specifying $R$ is equivalent to specifying $P$ 
since we can recover $R$ as $\Z[p^{-1}\vert p\in P]$.

Bousfield and Kan defined in \cite{bousfield2009homotopy} for all $p \in P$ a $p$-completion functor $L_p \colon \An \to \An$
which are the universal functors which invert $H_*(-, \finfld{p})$-equivalences.
They also defined an $R$-localization functor $L_R \colon \An \to \An$,
which is the universal functor which inverts $H_*(-, R)$-equivalences.

They then proved the existence of arithmetic fracture squares:
\begin{thm*}[Bousfield-Kan]
    Let $X \in \An_*$ be a pointed nilpotent anima.
    Then there is a canonical cartesian square 
    \begin{center}
        \begin{tikzcd}
            X \arrow[r] \arrow[d] \arrow[dr, phantom,"\scalebox{1.5}{\color{black}$\lrcorner$}", near start, color=black] &\prod_{p \in P} L_p X \arrow[d] \\
            L_R X \arrow[r] &L_R \prod_{p \in P} L_p X.
        \end{tikzcd}
    \end{center}
\end{thm*}

The goal of this paper is to generalize this result to a large class of $\infty$-topoi.
Recall from \cite[Section 3]{mattis2024unstable} the definition of the unstable $p$-completion functor $L_p \colon \topos X \to \topos X$
for every $\infty$-topos $\topos X$ and any prime number $p \in \Z$.
In \cite{asok2022localization}, Asok-Fasel-Hopkins defined the unstable $R$-localization functor $L_R \colon \topos X \to \topos X$ on an $\infty$-topos $\topos X$.
We will review its construction in \cref{section:unstable}.

We will define the notion of a locally finite-dimensional cover of an $\infty$-topos in \cref{def:htpydim:cover}.
We say that an $\infty$-topos $\topos X$ admits a locally finite-dimensional cover if there exists such a cover.
This is the generality in which unstable fracture squares for arbitrary nilpotent sheaves exist.
Examples include any $\infty$-topos locally of homotopy dimension $\le n$ (e.g.\ any presheaf $\infty$-topos,
or the category of $G$-anima for a group $G$),
and topoi of the form $\ShvTop{\nis}{\smooth{S}}$ of Nisnevich sheaves 
on the category of smooth schemes over a qcqs base scheme $S$ of finite Krull-dimension,
which appear frequently in motivic homotopy theory.
For proofs that these $\infty$-topoi indeed admit locally finite-dimensional covers, see \cref{section:examples}.

Our main theorem is the following:
\begin{introthm}[\cref{lem:fracture-square:main-thm}]
    Let $\topos X$ be an $\infty$-topos that admits a locally finite-dimensional cover.
    Let $X \in \topos X_*$ be a nilpotent pointed sheaf (see \cite[Definition A.10]{mattis2024unstable} for the definition of nilpotence
    in an $\infty$-topos).
    Then there is a canonical cartesian square 
    \begin{center}
        \begin{tikzcd}
            X \arrow[r] \arrow[d] \arrow[dr, phantom,"\scalebox{1.5}{\color{black}$\lrcorner$}", near start, color=black] &\prod_{p \in P} L_p X \arrow[d] \\
            L_R X \arrow[r] &L_R \prod_{p \in P} L_p X.
        \end{tikzcd}
    \end{center}
\end{introthm}

From this, we can also immediately deduce the following:
\begin{introthm}
    Let $\topos X$ be an $\infty$-topos that admits a locally finite-dimensional cover.
    Let $f \colon X \to Y$ be a morphism of nilpotent pointed sheaves in $\topos X_*$.
    Then $f$ is an equivalence if and only if $f$ is a $p$-equivalence for all $p \in P$ (i.e.\ $L_p(f)$ is an equivalence for all $p \in P$)
    and $f$ is an $R$-local equivalence (i.e.\ $L_R(f)$ is an equivalence).
\end{introthm}

The assumptions on the $\infty$-topos can be relaxed if one requires the sheaf $X$ to be truncated:
\begin{introthm}[\cref{lem:fracture-square:truncated}]
    Let $\topos X$ be an $\infty$-topos with enough points.
    Let $X \in \topos X_*$ be a nilpotent pointed sheaf which is $n$-truncated for some $n$.
    Then there is a canonical cartesian square 
    \begin{center}
        \begin{tikzcd}
            X \arrow[r] \arrow[d] \arrow[dr, phantom,"\scalebox{1.5}{\color{black}$\lrcorner$}", near start, color=black] &\prod_{p \in P} L_p X \arrow[d] \\
            L_R X \arrow[r] &L_R \prod_{p \in P} L_p X.
        \end{tikzcd}
    \end{center}
\end{introthm}

\subsection*{Notation}
We will use the following symbols:
\begin{center}
    \begin{tabular}{ l|l }
        $R$ & a subring of $\Q$ \\ 
        $P$ & the set of prime numbers invertible in $R$ \\  
        $\An$ & the $\infty$-category of anima/spaces/$\infty$-groupoids \\
        $\Disc{\topos X}$ & the subcategory of $0$-truncated objects of an $\infty$-topos $\topos X$ \\
        $\Grp{\Cat C}$ & the group objects in a $1$-category $\Cat C$ with finite products \\ 
        $\AbObj{\Cat C}$ & the abelian group objects in a $1$-category $\Cat C$ with finite products \\
        $L_R$ & the (un)stable $R$-localization functors, see \cref{section:stable,section:unstable} \\
        $L_p$ & the (un)stable $p$-completion functors, see \cite[Sections 2 \& 3]{mattis2024unstable} \\
        $(-) \sslash p$ & the functor given by $X \mapsto \Cofib{X \xrightarrow{p} X}$ on a stable category
    \end{tabular}
\end{center}

\subsection*{Acknowledgement}
I want to thank Tom Bachmann for very helpfull discussions and reading a draft of this article.
I also thank Emma Brink and Timo Weiß for reading a draft of this article.

The author acknowledges support by the Deutsche Forschungsgemeinschaft
(DFG, German Research Foundation) through the Collaborative Research
Centre TRR 326 \textit{Geometry and Arithmetic of Uniformized Structures}, project number 444845124.

\newpage

\section{Stable \texorpdfstring{$R$}{R}-Localization}
\label{section:stable}
Let $\Cat A$ be a presentable additive $\infty$-category (\cite[Definition 2.6]{gepner2016universality}).
The two examples we have in mind are presentable stable $\infty$-categories 
and presentable abelian categories, which are additive by \cite[Proposition 2.8]{gepner2016universality}
(note that the homotopy category of a stable category is triangulated, so in particular additive).

\begin{defn}
    Let $T_{\Cat A}$ be the collection of morphisms 
    of the form \[m_{p, X} \colon X \xrightarrow{\Delta} \oplus_{i = 1}^p X \xrightarrow{\nabla} X,\] 
    i.e.\ given by 
    multiplication-by-$p$ on $X$, with $p \in P$ and $X \in \Cat A$.
    Let $\overline{T}_{\Cat A}$ be the strongly saturated class of morphisms generated by $T_{\Cat A}$.
    We call a morphism in $\overline{T}_{\Cat A}$ an \emph{$R$-local equivalence}.
\end{defn}

\begin{lem} \label{lem:stable:small-generation}
    The saturated class $\overline{T}_{\Cat A}$ is of small generation.
\end{lem}
\begin{proof}
   Since $\Cat A$ is presentable, it is in particular accessible, i.e.\ 
   it is generated under $\kappa$-filtered colimits 
   by a small set $K \subset \Cat A$ of $\kappa$-compact objects for some regular cardinal $\kappa$.
   In particular, we see that $\overline{T}_{\Cat A}$
   is generated by morphisms of the form $m_{p, X}$ with $p \in P$ and $X \in K$,
   which form a small set. 
\end{proof}

\begin{defn}
    An object $X \in \Cat A$ is called \emph{$R$-local}
    if $\Map{\Cat A}{f}{X}$ is an equivalence 
    for all $R$-local equivalences $f \colon Y \to Y'$.
\end{defn}

\begin{lem} \label{lem:stable:functor}
    There is a localization functor $L_R \colon \Cat A \to \Cat A$ (called \emph{$R$-localization})
    with the following properties:
    \begin{itemize}
        \item For a morphism $f \colon Y \to Y'$ the morphism $L_R(f)$ is an equivalence if and only if $f$ is an $R$-local equivalence,
        \item an object $X \in \Cat A$ is $R$-local if and only if the unit $X \xrightarrow{\simeq} L_R X$ is an equivalence, and 
        \item an object $X \in \Cat A$ is $R$-local if and only if $\Map{\Cat A}{f}{X}$ is an equivalence for all $f \in T_{\Cat A}$.
    \end{itemize}
\end{lem}
\begin{proof}
    This is \cite[Proposition 5.5.4.15]{highertopoi},
    using that $\overline{T}_{\Cat A}$ is of small generation by \cref{lem:stable:small-generation}.
\end{proof}

\begin{lem} \label{lem:stable:R-local-iff-p-invertible}
    An object $E \in \Cat A$ is $R$-local if and only if $m_{p, E}$ is an equivalence for all $p \in P$.
\end{lem}
\begin{proof}
    By \cref{lem:stable:functor} $E$ is $R$-local if and only if $\Map{\Cat A}{m_{p, X}}{E}$
    is an equivalence for all $p \in P$ and $X \in \Cat A$.
    Note that $\Map{\Cat A}{m_{p, X}}{E}$ is homotopic to $\Map{\Cat A}{X}{m_{p, E}}$, as both maps 
    are just multiplication by $p$ on the $\mathbb{E}_{\infty}$-group $\Map{\Cat A}{X}{E}$:
    Indeed, we have the following commutative diagram:
    \begin{center}
        \begin{tikzcd}
            \Map{\Cat A}{X}{E} \arrow[d, equal] \arrow[r, "\Delta"] &\Map{\Cat A}{\oplus_{i = 1}^p X}{E} \arrow[d, "\cong"] \arrow[r, "\nabla"] &\Map{\Cat A}{X}{E} \arrow[d, equal] \\
            \Map{\Cat A}{X}{E} \arrow[d, equal] \arrow[r, "\Delta"] &\oplus_{i = 1}^p \Map{\Cat A}{X}{E} \arrow[d, "\cong"] \arrow[r, "\nabla"] &\Map{\Cat A}{X}{E} \arrow[d, equal] \\
            \Map{\Cat A}{X}{E} \arrow[r, "\Delta"] &\Map{\Cat A}{X}{\oplus_{i = 1}^p E} \arrow[r, "\nabla"] &\Map{\Cat A}{X}{E}.
        \end{tikzcd}
    \end{center}
    Here, the composition of the top row is just $\Map{\Cat A}{m_{p, X}}{E}$,
    the composition of the bottom row is $\Map{\Cat A}{X}{m_{p, E}}$,
    and the composition of the middle row is multiplicatiion by $p$ on $\Map{\Cat A}{X}{E}$.
    Hence by the Yoneda lemma, $E$ is $R$-local if and only if $m_{p, E}$
    is an equivalence for all $p \in P$.
\end{proof}

We now want to describe the $R$-localization functor explicitly. If $R \neq \Z$, i.e.\ if $P \neq \emptyset$,
write $N$ for the set $P \times \N$, and $\operatorname{pr}_1 \colon N \to P$ for the first projection.
Enumerate $N$ via some bijective function $s \colon \N \to N$.
There is a functor $\underline{(-)}_\bullet \colon \Cat A \to \Fun{}(\N,\Cat A)$,
that sends an $A \in \Cat A$ to the $\N$-indexed diagram $\underline{A}_\bullet$
with $\underline{A}_n = A$ for all $n$,
and such that the map $\underline{A}_n \to \underline{A}_{n+1}$ is given by multiplication by $\operatorname{pr_1}(s(n))$.

\begin{lem} \label{lem:stable:R-loc-as-colim}
    Suppose that $R \neq \Z$, i.e.\ $P \neq \emptyset$.
    The $R$-localization functor $L_R$ is equivalent to the functor
    \[L'_R \coloneqq \colim{\N} \underline{(-)}_\bullet.\]
\end{lem}
\begin{proof}
    There is a natural morphism $\eta \colon \id{\Cat A} \to L'_R$ given by the canonical inclusion into 
    the colimit $A = A_0 \to \colimil{n} A_n = L'_R(A)$.
    It suffices to show that $\eta$ exhibits $L'_R$ as the $R$-localization functor,
    i.e.\ we have to show that $L'_R(A)$ is $R$-local for every $A$,
    and that the canonical map $\eta \colon A \to L'_R(A)$ is an $R$-local equivalence.

    Since strongly saturated classes 
    are stable under transfinite compositions (as they are closed under colimits in the arrow category),
    it immediately follows that $\eta$ is an $R$-local equivalence (because multiplication by a prime $p \in P$ 
    is an $R$-local equivalence by definition).

    We are left to show that $L'_R(A)$ is $R$-local.
    It suffices to show that $p$ is invertible on $L'_R(A)$ for every prime $p \in P$ by \cref{lem:stable:R-local-iff-p-invertible}.
    So fix a prime $p \in P$. 
    By \cite[Corollary 4.8]{gepner2016universality}, $\Cat A$ is canonically enriched over the category $\Sp_{\ge 0}$ 
    of connective spectra (or equivalently over the category of $\E_{\infty}$-groups).
    Thus, we see that $L'_R(A) \cong L'_R(\S) \otimes A$,
    where $L'_R(\S)$ is the same colimit applied to the sphere spectrum.
    It therefore suffices to show that $p$ is invertible on $L'_R(\S)$.
    Since the homotopy groups $\pi_n \colon \Sp_{\ge 0} \to \Ab$ commute with filtered colimits 
    and are jointly conservative, we are reduced to show that $p$ is invertible 
    on $L'_R(\pi_n(\S))$ for all $n \ge 0$.
    Thus, in particular we are reduced to the case that $\Cat A = \Ab$ is the 1-category of abelian groups.
        
    So we will construct an inverse to the multiplication by $p$ map $\psi \colon L'_R(A) \to L'_R(A)$,
    where $A$ is an abelian group.
    For each $n$ write $\iota_n \colon A \to L'_R(A)$ for the canonical maps into the colimit,
    and $m_j \colon A \to A$ for multiplication by $j \in \N$ on $A$.
    Note that we have 
    \begin{equation*}
        \psi \iota_n = \iota_n m_p.
    \end{equation*}
    We write $k(n)$ for the largest integer smaller then $n$ such that $\operatorname{pr_1}(s(k(n))) = p$.
    Since $p$ appears infinitely many times in $N$, $k(n) \to \infty$ as $n \to \infty$.
    In particular, there is a smallest $m(n) > n$ such that $k(m(n)) \ge n$.
    We let $\alpha(n) \coloneqq \prod_{i = n}^{k(m(n))} \operatorname{pr_1}(s(i))$.
    Let $\phi \colon L'_R(A) \to L'_R(A)$ be the map 
    given on the $n$-th component by 
    \begin{equation*}
        \phi \iota_n \coloneqq \iota_{k(n)} m_{\alpha(n)}
    \end{equation*}
    This is well defined:
    We have to check that $\phi \iota_n = \phi \iota_{n+1} m_{\operatorname{pr_1}(s(n))}$,
    which follows immediately from the definitions.
    Moreover, we see that 
    \begin{equation*}
        \iota_{k(n)} m_p m_{\alpha(n)} = \iota_{k(n)+1} m_{\alpha(n)} = \iota_{n}.
    \end{equation*}
    We now have to check that $\phi$ is an inverse to $\psi$.
    By the universal property of the colimit, we can check this on components:
    \begin{equation*}
        \phi \psi \iota_n = \phi \iota_n m_p = \iota_{k(n)} m_{\alpha(n)} m_p = \iota_n
    \end{equation*}
    and 
    \begin{align*}
        \psi \phi \iota_n = \psi \iota_{k(n)} m_{\alpha(n)} = \iota_{k(n)} m_p m_{\alpha(n)} = \iota_n.
    \end{align*}
    This proves the lemma.
\end{proof}

\begin{cor} \label{lem:stable:R-loc-commutes-with-seq-colim-pres-functors}
    Let $F \colon \Cat A \to \Cat B$ be an additive, sequential-colimit-preserving functor of presentable additive $\infty$-categories.
    Then $L_R F \cong F L_R$.
\end{cor}
\begin{proof}
    If $R = \Z$, then $L_R$ is just the identity functor, which clearly commutes with $F$.

    If $R \neq \Z$, we have the following diagram:
    \begin{center}
        \begin{tikzcd}
            \Cat A \arrow[d, "F"] \arrow[r, "\underline{(-)}_\bullet"] &\Fun{}(\N, \Cat A) \arrow[d, "F"] \arrow[r, "\colimil{}"] &\Cat A \arrow[d, "F"] \\
            \Cat B \arrow[r, "\underline{(-)}_\bullet"] &\Fun{}(\N, \Cat B) \arrow[r, "\colimil{}"] &\Cat B.
        \end{tikzcd}
    \end{center}
    The left diagram commutes by definition, and the right because $F$ preserves sequential colimits.
    By \cref{lem:stable:R-loc-as-colim} we know that the composition of the rows is the functor $L_R$.
    This immediately implies the corollary.
\end{proof}

If $\Cat D$ is a presentable stable $\infty$-category with a t-structure $\tstruct{\Cat D}$ \cite[Definition 1.2.1.4]{higheralgebra},
we will write $\heart{\Cat D} \coloneqq \tcon{\Cat D} \cap \tcocon{\Cat D}$ 
for the \emph{heart} of this t-structure (this is a abelian category, so in particular additive),
and $\pi_n \colon \Cat D \to \heart{\Cat D}$ for the homotopy object functors.
If the t-structure is accessible \cite[Definition 1.4.4.12]{higheralgebra} and compatible with filtered colimits, i.e.\ $\tcocon{\Cat D}$ is stable under filtered colimits,
then the heart $\heart{\Cat D}$ is itself presentable, see \cite[Remark 1.3.5.23]{higheralgebra}.
Moreover, in this case $\pi_n(-)$ commutes with filtered colimits, whence $L_R \pi_n(-) \cong \pi_n(L_R(-))$ by \cref{lem:stable:R-loc-commutes-with-seq-colim-pres-functors}.

The main result about stable $R$-localization is the following characterization via the homotopy groups:
\begin{prop} \label{lem:stable:R-local-on-htpy}
    Suppose that $\Cat D$ is a stable $\infty$-category with an accessible t-structure $\tstruct{\Cat D}$ compatible with filtered colimits,
    and that the t-structure is separated (i.e.\ $\pi_n(X) = 0$ for all $n$ already implies $X = 0$).

    A morphism $f \colon E \to F$ in $\Cat D$ is an $R$-local equivalence if and only if $\pi_n(f)$ is an $R$-local equivalence in $\heart{\Cat D}$ for all $n \in \Z$.

    Moreover, an object $E \in \Cat D$ is $R$-local if and only if $\pi_n(E)$ is $R$-local in $\heart{\Cat D}$ for all $n \in \Z$.

    In particular, $L_R E \cong F$ if and only if $L_R \pi_n(E) \cong \pi_n(F)$ via $f$ for all $n \in \Z$.
\end{prop}
\begin{proof}
    We have that $f$ is an $R$-local equivalence if and only if $L_R(f)$ is an equivalence.
    By separatedness, this is equivalent to $\pi_n(L_R f)$ being an equivalence for all $n$.
    As $\pi_n(L_R f) \cong L_R \pi_n(f)$ by \cref{lem:stable:R-loc-commutes-with-seq-colim-pres-functors}, 
    this is equivalent to $\pi_n(f)$ being an $R$-local equivalence for all $n$.
    This proves the first claim.

    For the second claim, note that $E$ is $R$-local if and only if $E \xrightarrow{\simeq} L_R(E)$ via the unit.
    By separatedness, this is equivalent to $\pi_n(E) \xrightarrow{\simeq} \pi_n(L_R E)$ for all $n$,
    which is equivalent to $\pi_n(E) \xrightarrow{\simeq} L_R \pi_n(E)$ being an equivalence for all $n$.
    But this just means that $\pi_n(E)$ is $R$-local for all $n$.
    This proves the proposition.
\end{proof}

\section{Unstable \texorpdfstring{$R$}{R}-Localization}
\label{section:unstable}
Let $\topos X$ be an $\infty$-topos \cite[Definition 6.1.0.2]{highertopoi}.
In this section, we will prove some results about the $R$-localization 
functor in $\topos X$, defined by Asok-Fasel-Hopkins \cite{asok2022localization}.
We will start by recalling its construction.

\begin{defn}
    For $p$ a prime write
    \[\rho_{p,1} \colon S^1 \cong B\Z \to B\Z \cong S^1\]
    induced by the morphism $p \colon \Z \to \Z$.
    Write $\rho_{p,n} \colon S^n \to S^n$ for the map 
    \[\rho_{p,n} \coloneqq \rho_{p,1} \wedge \id{S^{n-1}} \colon S^n \cong S^1 \wedge S^{n-1} \to S^1 \wedge S^{n-1} \cong S^n\]. 
\end{defn}

\begin{defn}
    Let $S$ be the collection of morphisms in $\topos X$ 
    given by
    \begin{equation*}
        \rho_{p,n,U} \coloneqq \rho_{p,n} \times \id{U} \colon S^n \times U \to S^n \times U \text{ for } p \in P, n \ge 1 \text{ and } U \in \topos X
    \end{equation*}
    and denote by $\overline{S}$ the strongly saturated class of morphisms of $\topos X$
    generated by $S$. 
    We call a morphism in $\overline{S}$ an \emph{$R$-local equivalence}.
\end{defn}

\begin{lem} \label{lem:localization:small-generation}
    The saturated class $\overline{S}$ is of small generation.
\end{lem}
\begin{proof}
    Let $(U_i)_{i \in I}$ be a small set of generators of $\topos X$,
    which exists since $\topos X$ is presentable.
    Then $\overline{S}$ is generated by the small set 
    of morphisms of the form $\rho_{p,n,U_i}$
    with $p \in P$, $n \ge 1$ and $i \in I$,
    since the product $- \times -$ commutes with colimits 
    in each variable (colimits are universal in $\infty$-topoi, see \cite[Proposition 6.1.0.1]{highertopoi}).
\end{proof}

\begin{defn}
    An object $X \in \topos X$ is called \emph{$R$-local}
    if $\Map{\topos X}{f}{X}$ is an equivalence 
    for all $R$-local equivalences $f \colon Y \to Y'$.
\end{defn}

\begin{lem} \label{lem:localization:functor}
    There is a localization functor $L_R \colon \topos X \to \topos X$ (called \emph{$R$-localization})
    with the following properties:
    \begin{itemize}
        \item Let $f \colon Y \to Y'$ be a morphism. Then $L_R(f)$ is an equivalence if and only if $f$ is an $R$-local equivalence,
        \item an object $X \in \topos X$ is $R$-local if and only if $X \cong L_R X$, and 
        \item an object $X \in \topos X$ is $R$-local if and only if $\Map{\topos X}{f}{X}$ is an equivalence for all $f \in S$.
    \end{itemize}
\end{lem}
\begin{proof}
    This is \cite[Proposition 5.5.4.15]{highertopoi},
    using that $\overline{S}$ is of small generation, see \cref{lem:localization:small-generation}.
\end{proof}
Our first goal is to show that $R$-local equivalences induce isomorphisms on $\pi_0$.
For this, we need the following well-known fact:
\begin{lem} \label{lem:localization:general-pi-0-eq}
    Let $\Cat C_1$ and $\Cat C_2$ be presentable $\infty$-categories, 
    and let $\overline{T}$ be a strongly saturated class of morphisms in $\Cat C_1$, 
    generated by a class $T$.
    If $F \colon \Cat C_1 \to \Cat C_2$ is a colimit-preserving functor such that $F(f)$ is an equivalence for every $f \in T$,
    then $F(f)$ is an equivalence for every $f \in \overline{T}$.
\end{lem}
\begin{proof}
    Let $T'$ be the class of morphisms $f$ such that $F(f)$ is an equivalence,
    i.e.\ $T'$ is the preimage under $F$ of the class of equivalences in $\Cat C_2$.
    It follows from \cite[Proposition 5.5.4.16]{highertopoi} that $T'$
    is strongly saturated 
    (note that the class of equivalences is strongly saturated by \cite[Example 5.5.4.9]{highertopoi}).
    In particular, since $T \subset T'$ by assumption,
    we conclude that $\overline{T} \subset T'$.
\end{proof}
As a corollary, we obtain:
\begin{cor} \label{lem:localization:pi-0-eq}
    An $R$-local equivalence $f \colon X \to Y$ induces an isomorphism
    $\pi_0(f) \colon \pi_0(X) \to \pi_0(Y)$.
\end{cor}
\begin{proof}
    The functor $\pi_0 \cong \tau_{\le 0} \colon \topos X \to \Disc{\topos X}$ preserves 
    colimits, as it is left adjoint to the inclusion.
    Whence by \cref{lem:localization:general-pi-0-eq}, 
    it suffices to show that $\pi_0(\rho_{p,n,U})$ is an isomorphism for all $p$, $n$ and $U$.
    But note that $\pi_0(\rho_{p,n,U})$ is the endomorphism on $\pi_0(S^n \times U) \cong \pi_0(S^n) \times \pi_0(U)$
    induced by $\rho_{p,n}$ on $\pi_0(S^n)$ and by the identity on $\pi_0(U)$.
    Since $S^n$ is connected, the result follows.
\end{proof}

Recall the following results from \cite{asok2022localization}:
\begin{lem} \label{lem:localization:R-loc-on-points}
    For a point $s$ of $\topos X$, there is a canonical equivalence $L_R s^* \cong s^* L_R$.

    If $\Cat S$ is a conservative family of points, then an object $X \in \topos X$ is $R$-local 
    if and only if $s^* X$ is $R$-local for all points $s \in \Cat S$.
    Similarly, a morphism $f \colon X \to Y$ in $\topos X$ is an $R$-local equivalence 
    if and only if $s^* f$ is an $R$-local equivalence for all $s \in \Cat S$.
\end{lem}
\begin{proof}
    The first statment is \cite[Proposition 2.3.6 (1)]{asok2022localization},
    note that their proof is working also if $\topos X$ is not $1$-localic.
    The last statment follows immediately from the first.
\end{proof}

We get the following as an immediate corollary:
\begin{cor} \label{lem:localization:R-loc-geom-enough-points}
    Let $f^* \colon \topos X \rightleftarrows \topos Y \colon f_*$
    be a geometric morphism of $\infty$-topoi.
    Suppose that $\topos Y$ has enough points.
    Then there is a canonical equivalence $L_R f^* \cong f^* L_R$.
\end{cor}
\begin{proof}
    We have a canonical map $f^* \to f^* L_R$.
    It suffices to prove that this map is an $R$-local equivalence,
    and that the right hand side is $R$-local.
    Since $\topos Y$ has enough points, by \cref{lem:localization:R-loc-on-points} 
    we can check both properties on stalks.
    So let $s^* \colon \topos Y \to \An$ be a point of $\topos Y$.

    Note that $s^* f^*$ is a point of $\topos X$.
    Therefore, since $\id{\topos X} \to L_R$ is an $R$-local equivalence,
    also $s^* f^* \to s^* f^* L_R$ is an $R$-local equivalence (again by \ref{lem:localization:R-loc-on-points}).
    Similarly, by the same lemma we also get that $s^* f^* L_R \cong L_R s^* f^*$ is $R$-local.
\end{proof}

Recall the notion of nilpotent sheaf in an $\infty$-topos \cite[Definition A.10]{mattis2024unstable}.
\begin{lem} \label{lem:localization:nilpotence-stability}
    If $\topos X$ has enough points, then the functor $L_R$ preserves nilpotent sheaves.
\end{lem}
\begin{proof}
    For a nilpotent sheaf $X \in \topos X$, apply \cite[Lemma 2.3.10]{asok2022localization} to the nilpotent morphism $X \to *$.
\end{proof}

Recall the following definition:
\begin{defn}
    Let $G \in \Grp{\Disc{\topos X}}$ be a sheaf of groups.
    We say that $G$ is $R$-local if for each $p \in P$ 
    the $p$-power map $x \mapsto x^p$ (considered as a map of sheaves of sets) is an isomorphism.    
\end{defn}

Asok-Fasel-Hopkins constructed an $R$-localization functor 
\begin{align*}
    L_R \colon \Grp{\Disc{\topos X}} &\to \Grp{\Disc{\topos X}} \\
    G &\mapsto \pi_1 (L_R BG).
\end{align*}
that exhibits the category of $R$-local nilpotent sheaves of groups as an exact localization 
of the category of nilpotent sheaves of groups \cite[Proposition 2.3.12 (1), (2)]{asok2022localization}.
If $A \in \Grp{\Disc{\topos X}}$ is abelian, then also $L_R A$ is abelian:
Since everything can bec checked on stalks, this follows from 
the corresponding result in $\An$, see \cite[V.2.1]{bousfield2009homotopy}.
In particular, we see that on sheaves of abelian groups, this functor agrees with the $R$-localization 
functor from \cref{section:stable} (as both are localization functors on the same class of local objects).

We will say that a map $f \colon G \to H$ of sheaves of groups
is an \emph{$R$-local equivalence} if and only if $L_R f \colon L_R G \to L_R H$ is an 
equivalence. If $G$ and $H$ are abelian, this agrees with the notion of $R$-local 
equivalence from \cref{section:stable}.

We get the following unstable analogue of \cref{lem:stable:R-local-on-htpy}:
\begin{prop} \label{lem:localization:R-local-on-htpy}
    Suppose that $\topos X$ has enough points.
    A morphism $f \colon X \to Y$ of pointed nilpotent sheaves in $\topos X_*$ is an $R$-local equivalence 
    if and only if $\pi_n(f)$ is an $R$-local equivalence for all $n \ge 1$.

    A pointed nilpotent sheaf $Y \in \topos X_*$ is $R$-local if and only if $\pi_n(Y)$ is $R$-local for all $n \ge 1$.

    In particular, $L_R X \cong Y$ via $f$ if and only if $L_R \pi_n(X) \cong \pi_n(Y)$ via $f$ for all $n \ge 1$.
\end{prop}
\begin{proof}
    We start with the proof of the second statement. 
    By \cite[Corollary 2.3.13 (2)]{asok2022localization}, we have a canonical equivalence $L_R \pi_n(Y) \cong \pi_n(L_R Y)$.
    If $Y$ is $R$-local, then $L_R \pi_n(Y) \cong \pi_n(L_R Y) \cong \pi_n(Y)$,
    i.e.\ $\pi_n(Y)$ is $R$-local.
    On the other hand, if $\pi_n(Y) \cong L_R \pi_n(Y)$ for all $n \ge 1$,
    then $\pi_n(Y) \cong \pi_n(L_R Y)$ for all $n \ge 1$.
    Note that $\pi_0(Y) = *$ since $Y$ is connected, and $\pi_0(L_R Y) = *$ by \cref{lem:localization:pi-0-eq}.
    We conclude by hypercompleteness that $Y \cong L_R Y$ (note that the topos is hypercomplete 
    since it has enough points, see \cite[Remark 6.5.4.7]{highertopoi}).

    For the first statement, we will again use the equivalences 
    $L_R \pi_n(X) \cong \pi_n(L_R X)$ and $L_R \pi_n(Y) \cong \pi_n(L_R Y)$ for all $n \ge 1$.
    In particular, for every $n \ge 1$ there is a commutative square 
    \begin{center}
        \begin{tikzcd}
            L_R \pi_n(X) \arrow[r, "L_R \pi_n(f)"] \arrow[d, "\cong"] &L_R \pi_n(Y) \arrow[d, "\cong"]\\
            \pi_n (L_R X) \arrow[r, "\pi_n(L_R f)"] & \pi_n(L_R Y).
        \end{tikzcd}
    \end{center}
    If $f$ is an $R$-local equivalence, then the bottom arrow is an equivalence.
    In particular, also the top arrow is an equivalence, i.e.\ $\pi_n(f)$ is an $R$-local equivalence.
    If on the other hand all the $\pi_n(f)$ are $R$-local equivalences,
    we conclude that for every $n \ge 1$ also $\pi_n(L_R f)$ is an equivalence.
    By nilpotence, $\pi_0(L_R X) = * = \pi_0(L_R Y)$.
    Thus, we conclude by hypercompleteness that $L_R f$ is an equivalence,
    i.e.\ $f$ is an $R$-local equivalence.

    The last statement follows by combining the two results.
\end{proof}

Asok-Fasel-Hopkins showed in \cite[Lemma 2.3.10]{asok2022localization} that for a morphism $f \colon Y \to K$ of pointed nilpotent sheaves with connected fiber,
then $L_R \Fib{f} \cong \Fib{L_R f}$.
We need the following strenghtening of their result:
\begin{lem} \label{lem:localization:fiber-cover}
    Suppose that $\topos X$ has enough points.
    Let $f \colon Y \to K$ be a morphism of pointed nilpotent sheaves $Y, K \in \topos X_*$.
    Then $\tau_{\ge 1} \Fib{L_R f} \cong L_R \tau_{\ge 1} \Fib{f}$.

    Moreover, if $\Fib{f}$ is connected, then $\Fib{L_R f} \cong L_R \Fib{f}$.
\end{lem}
\begin{proof}
    It follows from \cite[Lemma A.12]{mattis2024unstable} and \cref{lem:localization:nilpotence-stability}
    that $L_R \tau_{\ge 1} \Fib{f}$ and $\tau_{\ge 1} \Fib{L_R f}$ are nilpotent.
    Note that there is a natural comparison morphism $\alpha \colon \tau_{\ge 1} \Fib{f} \to \tau_{\ge 1} \Fib{L_R f}$.
    By \cref{lem:localization:R-local-on-htpy},
    it suffices to show that $\alpha$ induces for all $n \ge 1$ isomorphisms 
    \begin{equation*}
        L_R \pi_n (\tau_{\ge 1} \Fib{f}) \to \pi_n(\tau_{\ge 1} \Fib{L_R f}).
    \end{equation*}
    Note that the fiber sequences $\Fib{f} \to Y \to K$ and $\Fib{L_R f} \to L_R Y \to L_R K$
    give us a diagram of long exact sequences 
    \begin{center}
        \begin{tikzcd}
            \cdots \arrow[r] &\pi_{n+1} (K) \arrow[r] \arrow[d] &\pi_n(\Fib{f}) \arrow[r] \arrow[d] &\pi_n(Y) \arrow[r] \arrow[d] & \cdots \\
            \cdots \arrow[r] &\pi_{n+1} (L_R K) \arrow[r] &\pi_n(\Fib{L_R f}) \arrow[r] &\pi_n(L_R Y) \arrow[r] & \cdots.
        \end{tikzcd}
    \end{center}
    Since $L_R$ is an exact localization functor on $\Grp{\Disc{\topos X}}$,
    and since (again by \cref{lem:localization:R-local-on-htpy}) 
    we have canonical isomorphisms $L_R \pi_n(Y) \cong \pi_n(L_R Y)$ and $L_R \pi_n(K) \cong \pi_n(L_R K)$ for all $n \ge 1$,
    the $5$-lemma implies that we also get isomorphism $L_R \pi_n(\Fib{f}) \cong \pi_n(\Fib{L_R f})$.
    In particular, we conclude that 
    \begin{equation*}
        L_R \pi_n (\tau_{\ge 1} \Fib{f}) \cong L_R \pi_n(\Fib{f}) \cong \pi_n (\Fib{L_R f}) \cong \pi_n (\tau_{\ge 1} \Fib{L_R f}).
    \end{equation*}

    For the last statement, suppose that $\Fib{f}$ is connected.
    Thus, we get $L_R \Fib{f} \cong L_R \tau_{\ge 1} \Fib{f} \cong \tau_{\ge 1} \Fib{L_R f}$.
    Thus, it suffices to show that $\Fib{L_R f}$ is connected.
    This can again by seen from the long exact sequence:
    The morphism $\pi_1 L_R Y \to \pi_1 L_R K$ is surjective as it is isomorphic to the morphism $L_R \pi_1 Y \to L_R \pi_1 K$,
    which is surjective since $L_R$ is an exact functor, and the morphism $\pi_1 Y \to \pi_1 K$
    is surjective since $\Fib{f}$ is connected.
    As $\pi_0(L_R Y) = *$ by \cref{lem:localization:pi-0-eq}, this implies $\pi_0(L_R \Fib{f}) = *$.
\end{proof}

We show next that $R$-localization behaves well with respect to truncations and connected covers.
For this, we first need this simple lemma:
\begin{lem} \label{lem:localization:truncation-cover-geom-morphism}
    Let $f^* \topos X \rightleftarrows \topos Y \colon f_*$ be a geometric morphism.
    Let $X \in \topos X_*$ be a pointed sheaf, and $n \ge 0$.
    We then have equivalences 
    \begin{equation*}
        \tau_{\le n} f^* X \cong f^* \tau_{\le n} X, \tau_{\ge n} f^* X \cong f^* \tau_{\ge n} X, \text{ and } \pi_n(f^* X) \cong f^* \pi_n(X).
    \end{equation*}
\end{lem}
\begin{proof}
    The first equivalence is \cite[Proposition 6.3.1.9]{highertopoi}.
    For the second, note that there is by definition a fiber sequence 
    \begin{equation*}
        \tau_{\ge n} X \to X \to \tau_{\le n-1} X.
    \end{equation*}
    Because $f^*$ is left exact, it induces a fiber sequence
    \begin{equation*}
        f^* \tau_{\ge n} X \to f^* X \to f^* \tau_{\le n-1} X.
    \end{equation*}
    Since by the first point we have an equivalence $f^* \tau_{\le n-1} X \cong \tau_{\le n-1} X$,
    it also follows that $f^* \tau_{\ge n} X \cong \tau_{\ge n} f^* X$.
    The last equation follows from the first and left-exactness of $f^*$ because $\pi_n(X) \cong \Omega^n \tau_{\le n} X$.
\end{proof}

\begin{lem} \label{lem:localization:connected-cover-R-loc}
    Suppose that $\topos X$ has enough points.
    Suppose that $X \in \topos X_*$ is a pointed sheaf.
    Then there is a canonical equivalence 
    \begin{equation*}
        L_R \tau_{\ge 1} X \cong \tau_{\ge 1} L_R X.
    \end{equation*}
\end{lem}
\begin{proof}
    Since connective covers and $R$-localization commute with taking stalks (see \cref{lem:localization:truncation-cover-geom-morphism} and \cref{lem:localization:R-loc-on-points}),
    and since there are enough points,
    we can assume that $X$ is a pointed anima, so $X = \amalg_{j \in J} X_j$ decomposes as a coproduct 
    of connected spaces (this would not work in an arbitrary $\infty$-topos).
    Let $0 \in J$ be the index such that $X_0$ is the component containing the basepoint.
    Then $\tau_{\ge 1} X = X_0$.
    Since the colimit of $R$-local equivalences is again an $R$-local equivalence,
    \begin{equation*}
        X = \amalg_{j \in J} X_j \to \amalg_{j \in J} L_R X_j
    \end{equation*}
    is an $R$-local equivalence.

    We now show that $\amalg_{j \in J} L_R X_j$ is $R$-local.
    Let $U \in \An$ be an anima, $n \ge 1$ and $p \in P$.
    We have to show that $\rho_{p, n, U}$ induces an equivalence 
    $\Map{\An}{S^n \times U}{\amalg_{j \in J} L_R X_j} \to \Map{\An}{S^n \times U}{\amalg_{j \in J} L_R X_j}$.
    By writing $U$ as the disjoint union of its components, and pulling out the coproduct from the left side of the 
    mapping space, we may assume that $U$ is connected.
    But then, any map $S^n \times U \to \amalg_{j \in J} L_R X_j$ factors (uniquely) over some $L_R X_j$.
    The claim then follows because $L_R X_j$ is $R$-local.

    This now implies that $L_R X \cong \amalg_{j \in J} L_R X_j$.
    Note that by \cref{lem:localization:pi-0-eq}, $L_R X_0$ is connected.
    Therefore, we see that $\tau_{\ge 1} L_R X \cong L_R X_0 \cong L_R \tau_{\ge 1} X$.
    This proves the lemma.
\end{proof}

For the result about truncations we need the following stability property of nilpotent sheaves:
\begin{lem} \label{lem:localization:nilpotence-truncation}
    If $X \in \topos X_*$ is nilpotent, so is $\tau_{\le n} X$ for $n \ge 0$.
\end{lem}
\begin{proof}
    If $n = 0$, then $\tau_{\le n} X = *$ is nilpotent.
    Otherwise, we see that $\pi_1(X) \cong \pi_1(\tau_{\le n} X)$.
    In particular, the action of $\pi_1(\tau_{\le n} X)$ on $\pi_k(\tau_{\le n X})$
    is the same as the action of $\pi_1(X)$ on $\pi_k(X)$ if $k \le n$,
    or the trivial action on the trivial group if $k > n$.
    Thus, $\tau_{\le n} X$ is still nilpotent.
\end{proof}

\begin{lem} \label{lem:localization:truncation-R-loc}
    Suppose that $\topos X$ has enough points.
    Let $X \in \topos X_*$ be a pointed nilpotent sheaf.
    Then $L_R \tau_{\le n} X \cong \tau_{\le n} L_R X$ for all $n \ge 0$.
\end{lem}
\begin{proof}
    By \cref{lem:localization:nilpotence-stability} the sheaf $L_R X$ is nilpotent, 
    and by \cref{lem:localization:nilpotence-truncation} also $\tau_{\le n} X$ and $\tau_{\le n} L_R X$ are nilpotent.
    Whence by \cref{lem:localization:R-local-on-htpy} it suffices to show that 
    $L_R \pi_k(\tau_{\le n} X) \cong \pi_k(\tau_{\le n} L_R X)$ for all $k \ge 1$.
    If $k > n$ then both sides are $0$.
    On the other hand, if $k \le n$, then the result follows by using \cref{lem:localization:R-local-on-htpy} again,
    since we then get a chain of equivalences
    \begin{equation*}
        L_R \pi_k(\tau_{\le n} X) \cong L_R \pi_k(X) \cong \pi_k(L_R X) \cong \pi_k(\tau_{\le n} L_R X).
    \end{equation*}
\end{proof}

Let $\spectra X$ be the stabilization of $\topos X$,
with induced adjunction $\pSus \colon \topos X \rightleftarrows \spectra{X} \colon \Loop$ \cite[Section 1.4]{higheralgebra}.
The stabilization has a separated accessible t-structure
given by $\tcocon[-1]{\spectra{X}} = \set{E \in \spectra{X}}{\Loop E \cong *}$.
(This t-structure is always right-separated, see \cite[Lemma A.5]{mattis2024unstable}.
A similar proof shows that the t-structure is left-separated if $\topos X$ is hypercomplete,
e.g.\ if $\topos X$ has enough points \cite[Remark 6.5.4.7]{highertopoi}.)
Moreover, this t-structure is compatible with filtered colimits \cite[Proposition 1.3.2.7 (2)]{sag}.
We end this section by proving a comparison result between unstable and stable $R$-localization.
\begin{lem} \label{lem:localization:infinite-loop-space}
    Suppose that $\topos X$ has enough points.
    For a $1$-connective object $E \in \spectra{X}$,
    there is a canonical equivalence 
    \begin{equation*}
        L_R \Loop E \cong \Loop L_R E.
    \end{equation*}
\end{lem}
\begin{proof}
    There is a canonical map $\Loop E \to \Loop L_R E$.
    By \cref{lem:stable:R-local-on-htpy} the sheaf $L_R E$ is still $1$-connective,
    whence $\Loop E$ is nilpotent by \cite[Lemma A.11]{mattis2024unstable}.
    By \cref{lem:localization:R-local-on-htpy}, it therefore suffices to show that 
    $L_R \pi_n(\Loop E) \cong \pi_n(\Loop L_R E)$
    via the canonical map for all $n \ge 1$.
    As $\pi_n(\Loop E) \cong \pi_n(E)$,
    and similarly $\pi_n(\Loop L_R E) \cong \pi_n(L_R E)$, the lemma follows immediately 
    from \cref{lem:stable:R-local-on-htpy}.
\end{proof}

\section{Nilpotence-Stability of \texorpdfstring{$p$}{p}-Completions}
Let $\topos X$ be an $\infty$-topos.
Recall that for every $p \in P$ there exists a functor $L_p \colon \topos X \to \topos X$ which is the universal functor 
that inverts $p$-equivalences, i.e.\ morphisms $f \colon X \to Y$ in $\topos X$ 
such that $\pSus(f) \sslash p$ is an equivalence, see \cite[Section 3]{mattis2024unstable}.
In this section, we will prove that if $X \in \topos X_*$ is a nilpotent sheaf which is $n$-truncated for some $n$,
then $\tau_{\ge 1} \prod_p L_p X$ is nilpotent as well.
\begin{notation}
    We write $\prod_p$ for the product over all primes $p \in P$.
\end{notation}
We will start with the following well-known fact:  
\begin{lem} \label{lem:nilpotence-stability:connective-cover-limit}
    Let $\Cat I$ be a small $\infty$-category,
    and $X_\bullet \colon \Cat I \to \topos X_*$ be a diagram.
    For $n \ge 0$ there is an equivalence $\tau_{\ge n} \limil{i} \tau_{\ge n} X_i \cong \tau_{\ge n} \limil{i} X_i$.
\end{lem}
\begin{proof}
    Write temporarily $L \colon \topos X_{*, \ge n} \rightleftarrows \topos X_* \colon R$
    for the colocalization given by the $n$-connective cover,
    so that $LR \cong \tau_{\ge n}$.
    Note that $R$ commutes with limits (as a right adjoint),
    and that the unit $\id{} \to RL$ is an isomorphism (as $L$ is fully faithful).
    Thus we get equivalences
    \begin{align*}
        \tau_{\ge n} \limil{i} \tau_{\ge n} X_i 
        &\cong LR \limil{i} LR X_i \\
        &\cong L \limil{i} RLR X_i \\
        &\cong L \limil{i} R X_i \\
        &\cong LR \limil{i} X_i \\
        &\cong \tau_{\ge n} \limil{i} X_i.
    \end{align*}
\end{proof}

\begin{lem} \label{lem:nilpotence-stability:p-comp-prod-fiber}
    Let $f \colon Y \to K$ be a map of nilpotent sheaves in $\topos X_*$,
    where $Y$ is $n$-truncated for some $n \ge 0$, and $K$ is an Eilenberg-MacLane sheaf 
    in degree $n + 1$ (e.g.\ a layer of a principal refinement of the Postnikov tower of a nilpotent sheaf).
    Then there are equivalences 
    \begin{align*}
        \tau_{\ge 1} \prod_p L_p \Fib{Y \to K} 
        &\cong \tau_{\ge 1} \Fib{\prod_p L_p Y \to \prod_p L_p K} \\
        &\cong \tau_{\ge 1} \Fib{\tau_{\ge 1} \prod_p L_p Y \to \tau_{\ge 1} \prod_p L_p K}.
    \end{align*}
\end{lem}
\begin{proof}
    We calculate 
    \begin{align*}
        \tau_{\ge 1} \prod_p L_p \Fib{Y \to K}
        &\cong \tau_{\ge 1} \prod_p \tau_{\ge 1} \Fib{L_p Y \to L_p K} \\
        &\cong \tau_{\ge 1} \prod_p \Fib{L_p Y \to L_p K} \\
        &\cong \tau_{\ge 1} \Fib{\prod_p L_p Y \to \prod_p L_p K} \\
        &\cong \tau_{\ge 1} \Fib{\tau_{\ge 1} \prod_p L_p Y \to \tau_{\ge 1} \prod_p L_p K},
    \end{align*}
    where we used \cite[Proposition 3.19]{mattis2024unstable} in the first equivalence,
    \cref{lem:nilpotence-stability:connective-cover-limit} in the second and last equivalence,
    and the fact that limits commute with limits in the third equivalence.
\end{proof}

This now allows us to prove the stability property:
\begin{lem} \label{lem:nilpotence-stability:nilpotent-product}
    Let $X \in \topos X_*$ be $n$-truncated for some $n \ge 0$ and nilpotent.
    Then $\tau_{\ge 1} \prod_p L_p X$ is nilpotent as well.
\end{lem}
\begin{proof}
    Since $X$ is truncated and nilpotent, there is a sequence of truncated and nilpotent sheaves 
    $X = X_n \xrightarrow{p_n} X_{n-1} \to \dots \to X_0 = *$
    such that $p_k$ fits into a fiber sequence $X_k \to X_{k-1} \to K(A_k, n_k)$
    for some $A_k \in \AbObj{\Disc{\topos X}}$ and $n_k \ge 2$, such that $X_{k-1}$ is $(n_k - 1)$-connective,
    by \cite[Lemma A.15]{mattis2024unstable}.
    We will prove the result by induction on the minimal length of such a sequence,
    the case $n = 0$ is trivial.
    We will write $Y \coloneqq X_{n-1}$ and $K \coloneqq K(A_n, n_n)$,
    i.e.\ there is a fiber sequence $X \to Y \to K$.
    \cref{lem:nilpotence-stability:p-comp-prod-fiber} supplies us with an equivalence
    \begin{equation*}
        \tau_{\ge 1} \prod_p L_p X  
        \cong \tau_{\ge 1} \Fib{\tau_{\ge 1} \prod_p L_p Y \to \tau_{\ge 1} \prod_p L_p K}.
    \end{equation*}
    Now note that we have 
    \begin{equation*}
        \tau_{\ge 1} \prod_p L_p K \cong \tau_{\ge 1} \prod_p \pLoop \tau_{\ge 1} \Sigma^{n_n} L_p A_n \cong \tau_{\ge 1} \pLoop \prod_p \tau_{\ge 1} \Sigma^{n_n} L_p A_n,
    \end{equation*}
    where we used \cite[Corollary 3.18]{mattis2024unstable} and that $\pLoop$ commutes with limits as a right adjoint.
    Therefore, we see from \cite[Lemma A.11]{mattis2024unstable} that $\tau_{\ge 1} \prod_p L_p K$ is nilpotent because it 
    is a connected infinite loop sheaf.
    By induction, we see that $\tau_{\ge 1} \prod_p L_p Y$ is nilpotent.
    We therefore conclude by \cite[Lemma A.12]{mattis2024unstable} that $\tau_{\ge 1} \prod_p L_p X$
    is nilpotent as well.
\end{proof}

\section{Locally Finite-Dimensional Covers of \texorpdfstring{$\infty$}{infinity}-Topoi}
In this section, we introduce the notion of a locally finite-dimensional cover of an $\infty$-topos 
and prove some basic properties of such covers.

\begin{defn} \label{def:htpydim:cover}
    Let $\topos X$ be an $\infty$-topos.
    A \emph{locally finite-dimensional cover} $\col{\topos U_i}{i \in I}$ of $\topos X$ is collection of $\infty$-topoi $\topos U_i$ 
    and geometric morphisms 
    \begin{equation*}
        f^*_i \colon \topos X \rightleftarrows \topos U_i \colon f_{*,i}
    \end{equation*}
    that satisfy the following:
    \begin{enumerate}
        \item For $i \in I$, the $\infty$-topos $\topos U_i$ has enough points and is locally of homotopy dimension $\le n_i$ for some $n_i \ge 0$.
        \item The functor $f^*_i$ has a further left adjoint $f_{!,i}$ for each $i \in I$.
        \item Write $\Cat S_i$ for the collection of points of $\topos U_i$ for every $i \in I$.
            The collection
            \begin{equation*}
                \Cat S \coloneqq \bigcup_{i \in I} \set{s^* f^*_i \colon \topos X \to \An }{s^* \in \Cat S_i}
            \end{equation*} 
            of points of $\topos X$ is a conservative family (i.e.\ they jointly detect equivalences).
    \end{enumerate}
\end{defn}

\begin{lem} \label{lem:htpydim:jointly-conservative}
    Let $\topos X$ be an $\infty$-topos with a locally finite-dimensional cover $\col{\topos U_i}{i \in I}$.
    Then the functors $f^*_i \colon \topos X \to \topos U_i$ preserve all limits and colimits 
    and are jointly conservative.
\end{lem}
\begin{proof}
    Each functor $f^*_i$ has both a left adjoint $f_{!,i}$ and a right adjoint $f_{*,i}$, this immediately implies the first statement.
    The second statement is clear since every point $s^* \in \Cat S$ factors through some $f^*_i$ by definition,
    and the $s^* \in \Cat S$ are jointly conservative.
\end{proof}

\begin{lem} \label{lem:htpydim:postnikov-complete}
    Let $\topos X$ be an $\infty$-topos with a locally finite-dimensional cover $\col{\topos U_i}{i\in I}$.
    Then $\topos X$ is Postnikov-complete (and thus in particular hypercomplete).
\end{lem}
\begin{proof}
    Note that $\topos U_i$ is Postnikov-complete for every $i \in I$, see \cite[Proposition 7.2.1.10]{highertopoi}.
    Using \cite[Proposition 5.5.6.26]{highertopoi}, 
    we have to show that for every $X \in \topos X$ the canonical map $X \to \limil{n} \tau_{\le n} X$
    is an equivalence, and that for every Postnikov pretower $(X_n)_n$,
    the canonical map $\tau_{\le n} \limil{n} X_n \to X_n$ is an equivalence.
    Since the $f^*_i$ commute with truncations and limits and are jointly conservative by \cref{lem:localization:truncation-cover-geom-morphism,lem:htpydim:jointly-conservative},
    the result follows immediately from the Postnikov-completeness of $\topos U_i$.
    
    Hypercompleteness follows from the proof of \cite[Corollary 7.2.1.12]{highertopoi},
    which only uses Postnikov-completeness of the topos.
\end{proof}

\section{Locally Highly Connected Towers}
In this section, we introduce locally highly connected towers in a $\topos X$, subordinate to some locally finite-dimensional cover $\Cat U$.
We prove that $R$-localization commutes with limits along such towers,
and provide some examples.

\begin{defn}
    Let $\topos X$ be an $\infty$-topos with a locally finite-dimensional cover $\Cat U \coloneqq \col{\topos U_i}{i \in I}$,
    and let $(X_n)_n$ be an $\N$ indexed tower in $\topos X_*$ consisting of connected objects.
    We say that the tower is
    \begin{enumerate}
        \item \emph{nilpotent} if every sheaf $X_n$ is nilpotent,
        \item \emph{highly connected} if for every $k \ge 1$ there is a $N_k \ge 1$ such that $\pi_k(X_n) \cong \pi_k(X_{N_k})$ for all $n \ge N_k$,
        \item \emph{locally highly connected (subordinate to $\Cat U$)} if the tower $(f^*_i X_n)_n$ is highly connected for every $i \in I$.
    \end{enumerate}
\end{defn}

\begin{rmk}
    Let $(X_n)_n$ be a tower of connected objects in some $\infty$-topos $\topos X$.
    Then the tower is highly connected if and only if the connectivity of the fibers $\Fib{X_n \to X_{n-1}}$ goes to $\infty$ as $n \to \infty$.
\end{rmk}

\begin{exmpl} \label{lem:htpydim:highly-connected-geom-morphism}
    Let $f^* \colon \topos X \rightleftarrows \topos Y \colon f_*$ be a geometric morphism of $\infty$-topoi.
    Since $f^*$ commutes with homotopy objects,
    it follows that if $(X_n)_n$ is a highly connected tower in $\topos X$, then also $(f^* X_n)_n$ is highly connected.
\end{exmpl}

\begin{lem} \label{lem:htpydim:highly-connected-post-tower}
    Let $\topos X$ be an $\infty$-topos with a locally finite-dimensional cover $\Cat U \coloneqq \col{\topos U_i}{i\in I}$,
    and let $X \in \topos X_*$ be a connected pointed sheaf.
    The Postnikov tower $(\tau_{\le n} X)_n$ is locally highly connected.
    If $X$ is moreover nilpotent, then so is the Postnikov tower.
\end{lem}
\begin{proof}
    The Postnikov tower is highly connected, so in particular locally highly connected by \cref{lem:htpydim:highly-connected-geom-morphism}.
    If $X$ is nilpotent, then so is $\tau_{\le n} X$ for all $n$, see \cref{lem:localization:nilpotence-truncation}.
\end{proof}

\begin{lem} \label{lem:htpydim:highly-connected-limit-stalk}
    Let $\topos X$ be an $\infty$-topos locally of homotopy dimension $\le N$ for some $N$, 
    $s^* \colon \topos X \to \An$ be a point,
    and $(X_n)_n$ be a highly connected tower in $\topos X_*$.
    Then the natural map $s^* \limil{n} X_n \to \limil{n} s^* X_n$ is an equivalence.
\end{lem}
\begin{proof}
    We will repeatedly use that $\pi_n(s^*(-)) \cong s^*(\pi_n(-))$,
    see \cref{lem:localization:truncation-cover-geom-morphism}.
    We first show that the natural map induces isomorphisms 
    \begin{equation*}
        \pi_k(s^* \limil{n} X_n) \to \pi_k(\limil{n} s^* X_n)
    \end{equation*}
    for every $k \ge 0$ (for the induced basepoints on the respective objects).
    By assumption, for every $k \ge 0$ there is an $N_k \ge 1$ such that the transition maps induce 
    equivalences 
    \begin{equation*}
        \pi_k(s^* X_n) \cong s^* \pi_k(X_n) \to s^* \pi_k(X_{N_k})
    \end{equation*} 
    for all $n \ge N_k$ (in the case $k = 0$ we can take $N_0 = 1$, as the objects are connected by assumption).
    We may assume without loss of generality that $N_k \ge N_{k'}$ whenever $k > k'$ by taking the maximum.
    Thus, we get from e.g.\ \cite[Lemma 3.24]{mattis2024unstable} an equivalence 
    \begin{equation*}
        \pi_k(\limil{n} s^* X_n) \cong s^* \pi_k(X_{N_k}).
    \end{equation*}
    On the other hand, $\pi_k(s^* \limil{n} X_n) \cong s^* \pi_k(\limil{n} X_n)$.
    It therefore suffices to show that $\pi_k(\limil{n} X_n) \cong \pi_k(X_{N_k})$ via the projection.
    Because the $N_k$ are non-decreasing, this is equivalent to showing that 
    the map fiber $\Fib{\limil{n} X_n \to X_{N_k}}$ is $k$-connective. 
    Since the topos is generated under colimits by objects $U$ of homotopy dimension $\le N$,
    it suffices to prove that for any such $U$ the anima
    \begin{equation*}
        \Fib{\limil{n} X_n \to X_{N_k}}(U) \cong \Fib{\limil{n} X_n(U) \to X_{N_k}(U)}.
    \end{equation*}
    is $k$-connective,
    By assumption, the map $X_n \to X_{N_k}$ is $k$-connective, therefore,
    the map $X_n(U) \to X_{N_k}(U)$ is $k - N$-connective 
    as $U$ has homotopy dimension $N$ (see \cite[Lemma 7.2.1.7]{highertopoi}).
    This immediately implies that $\Fib{\limil{n} X_n(U) \to X_{N_k}(U)}$ is $k - N - 1$-connective,
    because sequential limits of anima can lower the connectivity by at most one (this follows from the Milnor sequence).
    The result follows by reindexing the $N_k$.

    The lemma now follows from Whitehead's theorem, if we can show that $s^* \limil{n} X_n$ (and thus also $\limil{n} s^* X_n$) 
    is connected.
    But we have seen above that $\pi_0(s^* \limil{n} X_n) \cong s^* \pi_0(\limil{n} X_n) \cong s^* \pi_0 (X_{N_0}) = *$,
    since by assumption $X_n$ is connected for every $n$.
\end{proof}

\begin{cor} \label{lem:htpydim:locally-highly-connected-limit-stalk}
    Let $\topos X$ be an $\infty$-topos with a locally finite-dimensional cover $\Cat U \coloneqq \col{\topos U_i}{i\in I}$,
    and $(X_n)_n$ be a locally highly connected tower subordinate to $\Cat{U}$.
    Let $s^* \in \Cat S$ be a point (i.e.\ the point $s^*$ factors over some $f^*_i$).
    Then the natural map $s^* \limil{n} X_n \to \limil{n} s^* X_n$ is an equivalence.
\end{cor}
\begin{proof}
    Write $s^* \cong s^*_i f^*_i$. The natural map factors as 
    \begin{equation*}
        s^*_i f^*_i \limil{n} X_n \to s^*_i \limil{n} f^*_i X_n \to \limil{n} s^*_i f^*_i X_n,
    \end{equation*}
    which are equivalences by \cref{lem:htpydim:jointly-conservative,lem:htpydim:highly-connected-limit-stalk}, 
    as $(f^*_i X_n)_n$ is highly connected by definition.
\end{proof}

We now prove that the $R$-localization commutes with limits along locally highly connected towers.
\begin{lem} \label{lem:htpydim:R-loc-highly-connected}
    Let $\topos X$ be an $\infty$-topos with a locally finite-dimensional cover $\Cat U \coloneqq \col{\topos U_i}{i \in I}$.
    If $(X_n)_n$ is a nilpotent locally highly connected tower subordinate to $\Cat U$,
    then so is $(L_R X_n)_n$.
\end{lem}
\begin{proof}
    First note that $(L_R X_n)_n$ consists of nilpotent objects by \cref{lem:localization:nilpotence-stability}.
    Let $i \in I$.
    By assumption, $(f^*_i X_n)_n$ is highly connected, i.e.\ for every $k \ge 0$
    the system $(\pi_k(f^*_i X_n))_n$ is eventually constant, say, for $n \ge N_k$ for some $N_k$.
    So pick $k \ge 0$, we have to show that the same is true for the system $(\pi_k(f^*_i L_R X_n))_n$.
    We claim that the same $N_k$ works: We have
    \begin{equation*} 
        \pi_k(f^*_i L_R X_n) 
        \cong \pi_k(L_R f^*_i X_n)
        \cong L_R \pi_k(f^*_i X_n),
    \end{equation*}
    where we used \cref{lem:localization:R-loc-geom-enough-points} in the first equivalence
    and \cref{lem:localization:R-local-on-htpy} for the second equivalence
    (note that $f^*_i X_n$ is nilpotent).
    But this is constant for $n \ge N_k$ by assumption.
\end{proof}

\begin{lem} \label{lem:htpydim:lim-of-R-local-equivs}
    Let $(f_n)_n \colon (X_n)_n \to (Y_n)_n$ be a morphism of nilpotent highly connected towers of anima.
    Suppose that $f_n$ is an $R$-local equivalence for all $n$.
    Then $\limil{n} f_n \colon \limil{n} X_n \to \limil{n} Y_n$ 
    is an $R$-local equivalence.
\end{lem}
\begin{proof}
    Note that by definition, all the $X_n$ and $Y_n$ are nilpotent.
    Choose integers $N_k > 0$ for every $k \ge 0$ as in the defintion of highly connected towers,
    by taking the maximum we may assume that the $N_k$ work for both $(X_n)_n$ and $(Y_n)_n$.
    We know that $\pi_k(\limil{n} X_n) \cong \pi_k(X_{N_k})$ and $\pi_k(\limil{n} Y_n) \cong \pi_k(Y_{N_k})$
    for all $k \ge 0$, see e.g.\ \cite[Lemma 3.24]{mattis2024unstable}.
    By \cref{lem:localization:R-local-on-htpy}, it suffices to show that 
    the map $\pi_k(\limil{n} f_n)$ is an $R$-local equivalence for every $k \ge 0$,
    but this map corresponds under the above equivalences to the map $\pi_k(f_{N_k}) \colon \pi_k(X_{N_k}) \to \pi_k(Y_{N_k})$,
    which is an $R$-local equivalence by the same lemma.
\end{proof}

\begin{prop} \label{lem:htpydim:R-loc-of-highly-connected}
    Let $\topos X$ be an $\infty$-topos with a locally finite-dimensional cover $\Cat U \coloneqq \col{\topos U_i}{i \in I}$,
    and $(X_n)_n$ be a nilpotent locally highly connected tower subordinate to $\Cat U$.
    Then $L_R \limil{n} X_n \cong \limil{n} L_R X_n$.
\end{prop}
\begin{proof}
    Since a limit of $R$-local objects is $R$-local,
    it suffices to show that the natural map $\limil{n} X_n \to \limil{n} L_R X_n$
    is an $R$-local equivalence.
    By \cref{lem:localization:R-loc-on-points} this can be checked on the conservative family of points $\Cat S$.
    So suppose that $s^* f_i^* \in \Cat S$ is such a point. 
    Note that both $(X_n)_n$ and $(L_R X_n)_n$ 
    are nilpotent locally highly connected, the first by assumption,
    the second by \cref{lem:htpydim:R-loc-highly-connected}.
    Therefore, the towers $(f_i^* X_n)_n$ and $(f_i^* L_R X_n)_n$ are nilpotent highly connected by definition.
    Consider the following commutative diagram: 
    \begin{center}
        \begin{tikzcd}
            s^* f_i^* \limil{n} X_n \arrow[d, "\cong"] \arrow[r] &s^* f_i^* \limil{n} L_R X_n \arrow[d, "\cong"] \\
            \limil{n} s^* f_i^* X_n \arrow[r] &\limil{n} s^* f_i^* L_R X_n. \\
        \end{tikzcd}
    \end{center}
    The vertical arrows are equivalences by \cref{lem:htpydim:jointly-conservative,lem:htpydim:highly-connected-limit-stalk}.
    It therefore suffices to show that the map $\limil{n} s^* f_i^* X_n \to \limil{n} s^* f_i^* L_R X_n$
    is an $R$-local equivalence in $\An$.
    Note that $s^* f_i^* X_n \to s^* f_i^* L_R X_n$ is an $R$-local equivalence by \cref{lem:localization:R-loc-on-points}.
    The lemma thus follows from \cref{lem:htpydim:lim-of-R-local-equivs},
    since the towers $(s^* f_i^* X_n)_n$ and $(s^* f_i^* L_R X_n)_n$ are clearly also nilpotent highly connected. 
\end{proof}

We end this section by proving that $p$-completion of a nilpotent sheaf can be computed on the Postnikov tower.
This is a reformulation of \cite[Theorem 3.27]{mattis2024unstable}, but in the context of an $\infty$-topos $\topos X$
with a locally finite-dimensional cover.

\begin{lem} \label{lem:htpydim:truncated-nilpotent-geom-mor-p-comp}
    Let $f^* \colon \topos X \rightleftarrows \topos Y \colon f_*$ be a geometric 
    morphism of $\infty$-topoi such that $f^*$ has a further left adjoint $f_!$.
    For a sheaf $X \in \topos X$ we have a canonical equivalence $L_p f^* X \cong f^* L_p X$.
\end{lem}
\begin{proof}
    There is a canonical map $f^* X \to f^* L_p X$,
    which is a $p$-equivalence by \cite[Lemma 3.11]{mattis2024unstable}.
    It therefore suffices to show that $f^* L_p X$ is $p$-complete.
    This follows formally since the further left adjoint $f_!$ preserves $p$-equivalences 
    by the same lemma.
\end{proof}

\begin{lem} \label{lem:htpydim:prod-p-compl-loc-high-con}
    Let $\topos X$ be an $\infty$-topos with a locally finite-dimensional cover $\Cat U \coloneqq \col{\topos U_i}{i \in I}$.
    Suppose that $X \in \topos X_*$ is a nilpotent sheaf.
    Then $(\tau_{\ge 1} \prod_{p \in P} L_p \tau_{\le n} X)_n$ is a nilpotent locally highly connected tower subordinate to $\Cat U$.
    The same is true for the tower $(L_p \tau_{\le n} X)_n$ for any prime $p$.
\end{lem}
\begin{proof}
    Note that for a prime $p$ the sheaf $L_p \tau_{\le n} X$ is 
    connected as $\tau_{\le n} X$ is by assumption, and $p$-completion preserves connected objects, see \cite[Lemma 3.12]{mattis2024unstable}.
    Thus, the second statement is just a special case of the first, with $P = \{p\}$ 
    The sheaves $\tau_{\ge 1} \prod_p L_p \tau_{\le n} X$ are nilpotent by \cref{lem:nilpotence-stability:nilpotent-product}.
    Choose $i \in I$.
    We have to show that $(f^*_i \tau_{\ge 1} \prod_p L_p \tau_{\le n} X)_n$
    is a highly connected tower.
    We have 
    \begin{equation*}
        f^*_i \tau_{\ge 1} \prod_p L_p \tau_{\le n} X 
        \cong \tau_{\ge 1} \prod_p L_p \tau_{\le n} f^*_i X,
    \end{equation*}
    where we used \cref{lem:localization:truncation-cover-geom-morphism,lem:htpydim:jointly-conservative,lem:htpydim:truncated-nilpotent-geom-mor-p-comp}.
    We can therefore reduce to the case that $\topos X$ is locally of homotopy dimension $\le N$,
    and have to show that $(\tau_{\ge 1} \prod_p L_p \tau_{\le n} X)_n$ is highly connected.

    So let $k \ge 1$. We have to find an $N_k \ge 1$ such that 
    \begin{equation*}
        \pi_k(\tau_{\ge 1} \prod_p L_p \tau_{\le n} X) \cong \pi_{k}(\tau_{\ge 1} \prod_p L_p \tau_{\le N_k} X)
    \end{equation*}
    for all $n \ge N_k$.
    We claim that $N_k \coloneqq k + N + 3$ works.
    We prove the result by induction on $n \ge N_k$,
    the case $n = N_k$ being trivial.
    Since the sheaf $X$ is nilpotent by assumption,
    we can find a principal refinement of the Postnikov tower, see \cite[Lemma A.15]{mattis2024unstable}.
    We can therefore find an $m_n \ge 1$ and fiber sequences $X_{n, m} \to X_{n, m-1} \to K(A_m, n+1)$ for each $1 \le m \le m_n$
    such that $X_{n,0} = \tau_{\le n-1}X$, $X_{n,m_n} = \tau_{\le n}$, all the $X_{n, m}$ are $n$-truncated and $A_m$ abelian group onjects in $\Disc{\topos X}$.
    By induction on $m$, it therefore suffices to show that $\pi_k(\tau_{\ge 1} \prod_p L_p X_{n,m}) \cong \pi_{k}(\tau_{\ge 1} \prod_p L_p X_{n,m-1})$.
    We write $Z \coloneqq X_{n,m}$, $Y \coloneqq X_{m,m-1}$ and $K \coloneqq K(A_l, n+1)$,
    hence there is a fiber sequence $Z \to Y \to K$.
    We have 
    \begin{align*}
        \tau_{\ge 1} \prod_p L_p Z 
        &\cong \tau_{\ge 1} \prod_p L_p \Fib{Y \to K} \\
        &\cong \tau_{\ge 1} \prod_p \tau_{\ge 1} \Fib{L_p Y \to L_p K} \\
        &\cong \tau_{\ge 1} \prod_p \Fib{L_p Y \to L_p K} \\
        &\cong \tau_{\ge 1} \Fib{\prod_p L_p Y \to \prod_p L_p K} \\
        &\cong \tau_{\ge 1} \Fib{\tau_{\ge 1} \prod_p L_p Y \to \tau_{\ge 1} \prod_p L_p K},
    \end{align*}
    where we used \cite[Proposition 3.20]{mattis2024unstable} in the second equivalence,
    \cref{lem:nilpotence-stability:connective-cover-limit} in the third and last equivalences,
    and that limits commute with limits in the fourth equivalence.
    By the long exact sequence in homotopy, it therefore suffices to show that the connectivity of $\tau_{\ge 1} \prod_p L_p K$ is 
    at least $k + 2$. By \cite[Corollary 3.18]{mattis2024unstable} and \cref{lem:nilpotence-stability:connective-cover-limit} there is an equivalence 
    \begin{equation*}
        \tau_{\ge 1} \prod_p L_p K \cong \tau_{\ge 1} \prod_p \pLoop \limil{k} \Sigma^n HA \sslash p^k.
    \end{equation*}
    It therefore is enough to show that $\prod_p \pLoop \limil{k} \Sigma^n HA \sslash p^k$
    is $k + 2$-connective.
    It suffices to show that $(\prod_p \pLoop \limil{k} \Sigma^n HA \sslash p^k)(U)$
    is $k + 2$-connective for every $U$ of homotopy dimension $\le N$
    (as they generate the topos under colimits by definition).
    But we have 
    \begin{equation*}
        (\prod_p \pLoop \limil{k} \Sigma^n HA \sslash p^k)(U) \cong \prod_p \pLoop \limil{k} \Sigma^n ((HA)(U)) \sslash p^k,
    \end{equation*}
    as evaluation of sheaves commutes with limits and infinite loop spaces (note that by stability, $(-) \sslash p^k$ is the shift of a fiber).
    As $U$ is of homotopy dimension $\le N$, we see that $\Sigma^n (HA)(U)$ is $n-N$-connective (see e.g.\ \cite[Lemma 7.2.1.7]{highertopoi},
    applied to the overtopos $\topos X_{/U}$).
    Therefore, $\prod_p \pLoop \limil{k} \Sigma^n (HA)(U) \sslash p^k$
    is $n - N - 1$-connective: $(-) \sslash p^k$ is a colimit and therefore preserves connectivity.
    On the other hand, products of anima perserver connectivity, and sequential limits of anima 
    can lower the connectivity by one by the Milnor sequence.
    But as $n \ge N_k = k + N + 3$,
    we get that the connectivity of the anima in question is at least $(k + N + 3) - N - 1 = k + 2$.
    This proves the lemma.
\end{proof}

\begin{prop} \label{lem:htpydim:p-comp-of-post-tower}
    Let $\topos X$ be an $\infty$-topos with a locally finite-dimensional cover $\Cat U \coloneqq \col{\topos U_i}{i \in I}$,
    and $X \in \topos X_*$ be a nilpotent pointed sheaf.
    Then we have a canonical equivalence 
    $\prod_p L_p X \cong \limil{n} \prod_p L_p \tau_{\le n} X$.
\end{prop}
\begin{proof}
    By \cref{lem:htpydim:postnikov-complete}, $\topos X$ is Postnikov-complete,
    and we thus have to show that $\prod_p L_p \limil{n} \tau_{\le n} X \cong \limil{n} \prod_p L_p \tau_{\le n} X$.
    Since products commute with limits, it suffices to prove the result for a single $p \in P$,
    i.e.\ we have to prove that $L_p \limil{n} \tau_{\le n} X \cong \limil{n} L_p \tau_{\le n} X$.
    Since the right hand side is $p$-complete as a limit of $p$-complete objects,
    it therefore suffices to prove that the canonical map $\limil{n} \tau_{\le n} X \to \limil{n} L_p \tau_{\le n} X$
    is a $p$-equivalence.
    This can be checked on the conservative family of points $\Cat S$, see \cite[Lemma 3.11]{mattis2024unstable}.
    Let $s^* \in \Cat S$, i.e.\ $s^*$ is a point of $\topos X$ that factors through some $f_i^* \colon \topos X \to \topos U_i$.
    By \cref{lem:htpydim:locally-highly-connected-limit-stalk} we have equivalences
    \begin{equation*}
        s^* \limil{n} \tau_{\le n} X \cong \limil{n} s^* \tau_{\le n} X
    \end{equation*}
    and 
    \begin{equation*}
        s^* \limil{n} L_p \tau_{\le n} X \cong \limil{n} s^* L_p \tau_{\le n} X,
    \end{equation*}
    as both $(\tau_{\le n} X)_n$ and $(L_p \tau_{\le n} X)_n$ are locally highly connected subordinate to $\Cat U$, by \cref{lem:htpydim:highly-connected-post-tower,lem:htpydim:prod-p-compl-loc-high-con}, respectively.
    By definition and \cite[Lemma 3.11]{mattis2024unstable}, for each $n$ the map $s^* \tau_{\le n} X \to s^* L_p \tau_{\le n} X$ is a $p$-equivalence of anima.
    Thus, also $\limil{n} s^* \tau_{\le n} X \to \limil{n} s^* L_p \tau_{\le n} X$ is a $p$-equivalence by \cite[Lemma A.31]{mattis2024unstable}.
    This proves the proposition.
\end{proof}

\section{Stable Arithmetic Fracture Squares}
Let $\Cat D$ be a stable presentable $\infty$-category.
In this section, we will prove the stable analog of the main theorem, \cref{lem:fracture-square:main-thm}.

\begin{prop} \label{lem:stable:abstract-fracture-square}
    Suppose that there is a small set $I$,
    and for every $i \in I$ an exact localization functor $L_i \colon \Cat D \to \Cat D$.
    Let $L' \colon \Cat D \to \Cat D$ be another exact localization functor.
    In particular, there are natural transformations 
    $\alpha_i \colon \id{\Cat D} \to L_i$ and $\alpha' \colon \id{\Cat D} \to L'$.
    Suppose further that we have the following compatibilities:
    \begin{enumerate}[label=(\alph*)]
        \item \label{cond:fracture-system:comp-0} $L_j L' \cong 0$ for all $j \in I$.
        \item \label{cond:fracture-system:comp-Li-Lj} $L_j \prod_i \alpha_i \colon L_j \to L_j \prod_i L_i$ is an equivalence for all $j \in I$. 
        \item \label{cond:fracture-system:jointly-cons} $\{L'\} \cup \set{L_j}{j \in I}$ is a conservative family of functors.
    \end{enumerate}
    There is a cartesian square of functors 
    \begin{center}
        \begin{tikzcd}
            \id{\Cat D} \arrow[r, "\prod_i \alpha_i"] \arrow[d, "\alpha'"] &\prod_i L_i \arrow[d, "\alpha' \prod_i L_i"] \\
            L' \arrow[r, "L' \prod_i \alpha_i"] &L' \prod_i L_i.
        \end{tikzcd}
    \end{center}
\end{prop}
\begin{proof}
    Note that the square is clearly commutative.
    Let $G$ be the pullback of the span $L' \to L' \prod_i L_i \leftarrow \prod_i L_i$.
    There is a canonical map $f \colon \id{\Cat D} \to G$.
    We have to show that $f$ is an equivalence.
    Using \ref{cond:fracture-system:jointly-cons},
    it suffices to show that $L'f$ and the $L_j f$ are equivalences.

    We first show that $L'f$ is an equivalence.
    Since $L'$ is exact, applying it yields the following diagram where the square is cartesian:
    \begin{center}
        \begin{tikzcd}
            L' \arrow[rrrd, bend left, "L' \prod_j \alpha_j"] \arrow[ddr, bend right, "L' \alpha'"'] \arrow[rd, "L'f"] &&&\\
            &L' G \arrow[rr, "pr_1"] \arrow[d, "pr_2"'] &&L' \prod_i L_i \arrow[d, "L' \alpha' \prod_i L_i"] \\
            &L'L' \arrow[rr, "L'L' \prod_i \alpha_i"] &&L'L' \prod_i L_i.
        \end{tikzcd}
    \end{center}
    Since $L'$ is a localization functor, $L' \alpha'$ is an equivalence.
    Thus, in particular, the right vertical morphism is an equivalence.
    Since the inner square is cartesian, we see that also $pr_2$ is an equivalence.
    Hence, we conclude that $L'f$ is an equivalence.

    Now fix $j \in I$. We argue similarly that $L_j f$ is an equivalence:
    First, apply $L_j$ to get the following diagram, where again the inner square is cartesian:
    \begin{center}
        \begin{tikzcd}
            L_j \arrow[rrrd, bend left, "L_j \prod_i \alpha_i"] \arrow[ddr, bend right, "L_j \alpha'"'] \arrow[rd, "L_jf"] &&&\\
            &L_j G \arrow[rr, "pr_1"] \arrow[d, "pr_2"'] &&L_j \prod_i L_i \arrow[d, "L_j \alpha' \prod_i L_i"] \\
            &L_jL' \arrow[rr, "L_jL' \prod_i \alpha_i"] &&L_jL' \prod_i L_i.
        \end{tikzcd}
    \end{center}
    We see from \ref{cond:fracture-system:comp-0} that the bottom corners in the diagram are both $0$.
    In particular, they are equivalent, and therefore $pr_1$ is also an equivalence.
    The morphism $L_j \prod_i \alpha_i$ is an equivalence by \ref{cond:fracture-system:comp-Li-Lj}.
    Therefore we conclude that $L_j f$ is an equivalence.
    This proves the proposition.
\end{proof}

The stable $p$-completion functor $L_p(-) \coloneqq \limil{n} (-) \sslash p^n \colon \Cat D \to \Cat D$
is the universal functor which inverts $p$-equivalences (i.e.\ maps $f \colon X \to Y$ such that $f \sslash p$ is an equivalence),
see e.g.\ \cite[Section 2]{mattis2024unstable}.
This is an exact localization functor, write $\alpha_p \colon \id{\Cat D} \to L_p$ for the unit of the associated adjunction.
We will also write $\alpha_R \colon \id{\Cat D} \to L_R$ for the unit of the associated adjunction of the exact localization functor $L_R \colon \Cat D \to \Cat D$.

\begin{prop} \label{lem:stable:arithmetic-fracture-square}
    There is a cartesian square of functors 
    \begin{center}
        \begin{tikzcd}
            \id{\Cat D} \arrow[r, "\prod_p \alpha_p"] \arrow[d, "\alpha_R"] &\prod_p L_p \arrow[d, "\alpha_R \prod_p L_p"] \\
            L_R \arrow[r, "L_R \prod_p \alpha_p"] &L_R \prod_p L_p.
        \end{tikzcd}
    \end{center}
\end{prop}
\begin{proof}
    We want to apply \cref{lem:stable:abstract-fracture-square}
    with $I = P$, $L_i = L_p$ and $L' = L_R$.
    Thus, we have to show that $L_\ell L_R \cong 0$ for every prime $\ell \in P$,
    that $L_\ell \prod_p \alpha_p$ is an equivalence for every prime $\ell \in P$,
    and that the $L_\ell$ and $L_R$ are jointly conservative.

    We first show that $L_\ell \prod_p \alpha_p$ is an equivalence for every prime $\ell \in P$.
    So fix a prime $\ell$.
    Using \cite[Lemma 2.6]{mattis2024unstable}, it suffices to show that $\prod_p \alpha_p$ is an $\ell$-equivalence,
    i.e.\ that $(\prod_p \alpha_p) \sslash \ell \colon \id{\Cat D} \sslash \ell \to (\prod_p L_p) \sslash \ell \cong \prod_p (L_p \sslash \ell)$ is an equivalence.
    Note that $L_p \sslash \ell \cong 0$ for every prime $p \neq \ell$ since 
    $\ell$ is invertible on $L_p$.
    Therefore, $(\prod_p \alpha_p) \sslash \ell \cong \alpha_\ell \sslash \ell$,
    which is an equivalence by definition.

    We now show that $L_\ell L_R \cong 0$ for every prime $\ell \in P$.
    By \cref{lem:stable:R-local-iff-p-invertible}, multiplication by $\ell$ is an equivalence on $L_R$ for all $\ell \in P$.
    In particular, we get that $L_\ell L_R \cong 0$ from the description of $L_\ell$ as 
    the $\ell$-adic limit $\limil{k} (-) \sslash \ell^k$, see e.g.\ \cite[Lemma 2.5]{mattis2024unstable}.

    We are left to show that the functors $L_\ell$ and $L_R$ are jointly conservative.
    By stability, it is enough to show that if an object $X$ is $L_\ell$-acyclic for every $\ell \in P$ and $L_R$-acyclic,
    then $X \cong 0$.
    So suppose that $L_R X \cong 0 \cong L_\ell X$ for every $\ell \in P$.
    We deduce that $X \sslash \ell \cong (L_\ell X) \sslash \ell \cong 0$,
    i.e.\ multiplication by $\ell$ is invertible on $X$ for every $\ell \in P$.
    In particular, $X$ is $R$-local, by \cref{lem:stable:R-local-iff-p-invertible}.
    Therefore, $X \cong L_R X \cong 0$.
    This proves the proposition.
\end{proof}

\begin{cor} \label{lem:stable:arithmetic-fracture-square-applied-to-X}
    For every $X \in \Cat D$ the canonical square 
    \begin{center}
        \begin{tikzcd}
            X \arrow[r] \arrow[d] &\prod_p L_p X \arrow[d] \\
            L_R X \arrow[r] &L_R \prod_p L_p X
        \end{tikzcd}
    \end{center}
    is cartesian.
\end{cor}

We also need a version of the above corollary where we take $n$-connective covers of the $p$-completions.
For this, we need the following very general lemma:
\begin{lem} \label{lem:stable:cartesian-square-connective-cover}
    Suppose that $\Cat D$ is equipped with a t-structure $\tstruct{\Cat D}$,
    and that there is a cartesian square 
    \begin{center}
        \begin{tikzcd}
            X \arrow[d] \arrow[r] &Y \arrow[d] \\
            X' \arrow[r] &Y'.
        \end{tikzcd}
    \end{center}
    Suppose moreover that $X \in \tcon[n]{\Cat D}$ for some $n$.
    Then also the square 
    \begin{center}
        \begin{tikzcd}
            X \arrow[d] \arrow[r] &\tau_{\ge n} Y \arrow[d] \\
            \tau_{\ge n} X' \arrow[r] &\tau_{\ge n} Y'
        \end{tikzcd}
    \end{center}
    is cartesian.
\end{lem}
\begin{proof}
    Write $F$ for the limit of the span $\tau_{\ge n} X' \to \tau_{\ge n} Y' \leftarrow \tau_{\ge n} Y$.
    By stability, we know that there is a commutative diagram 
    \begin{center}
        \begin{tikzcd}
            F \arrow[d] \arrow[r] &\tau_{\ge n} (X' \oplus Y) \arrow[r]\arrow[d] &\tau_{\ge n} Y' \arrow[d] \\
            X \arrow[d] \arrow[r] &X' \oplus Y \arrow[d] \arrow[r] &Y' \arrow[d] \\
            \Cofib{F \to X} \arrow[r] &\tau_{\le n-1}(X' \oplus Y) \arrow[r] &\tau_{\le n-1} Y',
        \end{tikzcd}
    \end{center}
    where the rows and columns are co/fiber sequences. 
    To see that the map $F \to X$ is an equivalence, it therefore suffices to show that 
    \begin{equation*}
        \tau_{\le n-1} (X' \oplus Y) \to \tau_{\le n-1} Y'
    \end{equation*}
    is an equivalence. This follows from the assumption that $X \cong \Fib{X' \oplus Y \to Y'}$ is $n$-connective.
\end{proof}

\begin{cor} \label{lem:stable:arithmetic-fracture-square-connective-cover}
    Suppose that $\Cat D$ is equipped with a t-structure $\tstruct{\Cat D}$,
    which is accessible, separated and compatible with filtered colimits.
    Suppose that $X \in \tcon[n]{\Cat D}$.
    Then there is a canonical square 
    \begin{center}
        \begin{tikzcd}
            X \arrow[r] \arrow[d] &\tau_{\ge n} \prod_p L_p X \arrow[d] \\
            L_R X \arrow[r] &\tau_{\ge n} L_R \prod_p L_p X
        \end{tikzcd}
    \end{center}
    which is cartesian.
\end{cor}
\begin{proof}
    We get a canonical cartesian square 
    \begin{center}
        \begin{tikzcd}
            X \arrow[r] \arrow[d] &\prod_p L_p X \arrow[d] \\
            L_R X \arrow[r] &L_R \prod_p L_p X
        \end{tikzcd}
    \end{center}
    from \cref{lem:stable:arithmetic-fracture-square-applied-to-X}.
    Note that since $X$ is $n$-connective by assumption,
    and $L_R X$ is $n$-connective by \cref{lem:stable:R-local-on-htpy},
    it follows from \cref{lem:stable:cartesian-square-connective-cover} that
    \begin{center}
        \begin{tikzcd}
            X \arrow[r] \arrow[d] &\tau_{\ge n} \prod_p L_p X \arrow[d] \\
            L_R X \arrow[r] &\tau_{\ge n} L_R \prod_p L_p X
        \end{tikzcd}
    \end{center}
    is cartesian.
\end{proof}

\section{Unstable Arithmetic Fracture Squares}
Let $\topos X$ be an $\infty$-topos with enough points.
In this section, we will prove the main theorem.
We need the following simple lemma:
\begin{lem} \label{lem:fracture-square:cartesian-fiber}
    Let
    \begin{center}
        \begin{tikzcd}
            X \arrow[r] \arrow[d] &Y \arrow[r] \arrow[d] &Z \arrow[d, "h"] \\
            X' \arrow[r] &Y' \arrow[r] &Z'
        \end{tikzcd}
    \end{center}
    be a commutative diagram in a pointed $\infty$-category $\Cat C$.
    If the rows are fiber sequences and $h$ is an equivalence, then 
    the left square is cartesian.
\end{lem}
\begin{proof}
    We have $X \cong Y \times_Z * \cong Y \times_{Z'} * \cong Y \times_{Y'} Y' \times_{Z'} * \cong Y \times_{Y'} X'$,
    where we used in the first and last equivalence that the rows are fiber sequences, and in the second equivalence that $h$ 
    is an equivalence.
\end{proof}

\begin{cor} \label{lem:fracture-square:connective-cover-cartesian}
    Let $f \colon X \to Y$ be an $R$-local equivalence in $\topos X_*$.
    Then the square 
    \begin{center}
        \begin{tikzcd}
            \tau_{\ge 1} X \arrow[r] \arrow[d, "\tau_{\ge 1} f"] &X \arrow[d, "f"] \\
            \tau_{\ge 1} Y \arrow[r] &Y
        \end{tikzcd}
    \end{center}
    is cartesian.
\end{cor}
\begin{proof}
    This follows immediately from \cref{lem:fracture-square:cartesian-fiber},
    since $\pi_0(f) \colon \pi_0(X) \to \pi_0(Y)$ is an equivalence by 
    \cref{lem:localization:pi-0-eq}.
\end{proof}

\begin{lem} \label{lem:fracture-square:canonical-squares}
    Let $X \in \topos X_*$ be a connected sheaf.
    Then there are natural commutative squares 
    \begin{equation}
        \label{square:canonical}
        \begin{tikzcd}
            X \arrow[r] \arrow[d] &\prod_p L_p X \arrow[d] \\
            L_R X \arrow[r] &L_R \prod_p L_p X
        \end{tikzcd}
    \end{equation}
    and 
    \begin{equation}
        \label{square:connected-canonical}
        \begin{tikzcd}
            X \arrow[r] \arrow[d] &\tau_{\ge 1} \prod_p L_p X \arrow[d] \\
            L_R X \arrow[r] &\tau_{\ge 1} L_R \prod_p L_p X.
        \end{tikzcd}
    \end{equation}
    The first square is cartesian if and only if the second square is cartesian.
\end{lem}
\begin{proof}
    The square \eqref{square:canonical} is given by the natural maps into the localization functors (i.e.\ the units of the associated adjunctions).
    By assumption $X$ is connected, and hence so is $L_R X$ by \cref{lem:localization:pi-0-eq}.
    Therefore, the square \eqref{square:canonical} factors as 
    \begin{center}
        \begin{tikzcd}
            X \arrow[r] \arrow[d] &\tau_{\ge 1} \prod_p L_p X \arrow[d] \arrow[r] &\prod_p L_p X \arrow[d] \\
            L_R X \arrow[r] &\tau_{\ge 1} L_R \prod_p L_p X \arrow[r] &L_R \prod_p L_p X,
        \end{tikzcd}
    \end{center}
    which proves the existence of the square \eqref{square:connected-canonical}.
    Since the right square in the above diagram is cartesian (see \cref{lem:fracture-square:connective-cover-cartesian}),
    the last claim follows from the pasting law for pullback squares (the dual of \cite[Lemma 4.4.2.1]{highertopoi}).
\end{proof}

\begin{lem} \label{lem:fracture-square:EM-space}
    For $A \in \AbObj{\Disc{\topos X}}$ and $n \ge 1$, the canonical squares
    \begin{equation}
        \label{square:EM:canonical}
        \begin{tikzcd}
            K(A, n) \arrow[r] \arrow[d] &\prod_p L_p K(A, n) \arrow[d] \\
            L_R K(A, n) \arrow[r] &L_R \prod_p L_p K(A, n)
        \end{tikzcd}
    \end{equation}
    and 
    \begin{equation}
        \label{square:EM:connected-canonical}
        \begin{tikzcd}
            K(A, n) \arrow[r] \arrow[d] &\tau_{\ge 1}\prod_p L_p K(A, n) \arrow[d] \\
            L_R K(A, n) \arrow[r] &\tau_{\ge 1} L_R \prod_p L_p K(A, n)
        \end{tikzcd}
    \end{equation}
    from \cref{lem:fracture-square:canonical-squares}
    are cartesian.
\end{lem}
\begin{proof}
    By \cref{lem:fracture-square:canonical-squares},
    it suffices to show that the square \eqref{square:EM:connected-canonical}
    is cartesian.
    We write $HA \in \spectra{X}$ for the Eilenberg-MacLane spectrum of $A$.
    Since $\Sigma^n HA$ is $1$-connective, we get the following cartesian square from 
    \cref{lem:stable:arithmetic-fracture-square-connective-cover}, using the fact that $\pLoop$ preserves limits:
    \begin{center}
        \begin{tikzcd}
            \pLoop \Sigma^n HA \arrow[r] \arrow[d] &\pLoop \tau_{\ge 1} \prod_p L_p \Sigma^n HA \arrow[d] \\
            \pLoop L_R \Sigma^n HA \arrow[r] &\pLoop \tau_{\ge 1} L_R \prod_p L_p \Sigma^n HA.
        \end{tikzcd}
    \end{center}
    Note that we have $K(A, n) \cong \pLoop \Sigma^n HA$.
    Moreover, we have seen $\pLoop L_R E \cong L_R \pLoop E$ for $1$-connective sheaves of spectra $E$ in \cref{lem:localization:infinite-loop-space}.
    It follows from \cref{lem:stable:R-local-on-htpy} that $L_R \Sigma^n HA$ is $n$-connective, so in particular $1$-connective,
    whence $\pLoop L_R \Sigma^n HA \cong L_R K(A, n)$.
    Thus, we arrive at the following cartesian square (using that $L_R$ commutes with the $1$-connective cover functor, which follows easily from \cref{lem:stable:R-local-on-htpy}):
    \begin{center}
        \begin{tikzcd}
            K(A, n) \arrow[r] \arrow[d] &\tau_{\ge 1} \prod_p \pLoop L_p \Sigma^n HA \arrow[d] \\
            L_R K(A,n) \arrow[r] & L_R \tau_{\ge 1} \prod_p \pLoop L_p \Sigma^n HA.
        \end{tikzcd}
    \end{center}
    We now calculate 
    \begin{equation*}
        \tau_{\ge 1} \prod_p \pLoop L_p \Sigma^n HA 
        \cong \tau_{\ge 1} \prod_p \tau_{\ge 1} \pLoop L_p \Sigma^n HA
        \cong \tau_{\ge 1} \prod_p L_p K(A, n),
    \end{equation*}
    where we used \cref{lem:nilpotence-stability:connective-cover-limit} and \cite[Corollary 3.18]{mattis2024unstable}.
    Similarly, we calculate 
    \begin{equation*}
        L_R \tau_{\ge 1} \prod_p \pLoop L_p \Sigma^n HA
        \cong L_R \tau_{\ge 1} \prod_p L_p K(A, n)
        \cong \tau_{\ge 1} L_R \prod_p L_p K(A, n),
    \end{equation*}
    where we used the above equivalence and \cref{lem:localization:connected-cover-R-loc}.
    Thus, the above cartesian square is equivalent to 
    \begin{center}
        \begin{tikzcd}
            K(A, n) \arrow[r] \arrow[d] &\tau_{\ge 1} \prod_p L_p K(A, n) \arrow[d] \\
            L_R K(A,n) \arrow[r] &\tau_{\ge 1} L_R \prod_p L_p K(A, n),
        \end{tikzcd}
    \end{center}
    which proves the lemma.
\end{proof}

We need another lemma about pullback squares and connected covers.
This is an unstable version of \cref{lem:stable:cartesian-square-connective-cover}.
\begin{lem} \label{lem:fracture-sqaure:connected-cover-pullback-always}
    Let 
    \begin{center}
        \begin{tikzcd}
            X \arrow[r] \arrow[d] &Y \arrow[d] \\
            X' \arrow[r] &Y'
        \end{tikzcd}
    \end{center}
    be a cartesian square in $\topos X_*$.
    Suppose that $X$ is connected.
    Then 
    \begin{center}
        \begin{tikzcd}
            X \arrow[r] \arrow[d] &\tau_{\ge 1} Y \arrow[d] \\
            \tau_{\ge 1} X' \arrow[r] &\tau_{\ge 1} Y'
        \end{tikzcd}
    \end{center}
    is cartesian.
\end{lem}
\begin{proof}
    Let $F$ be the limit of the cospan $\tau_{\ge 1} X' \to \tau_{\ge 1} Y' \leftarrow \tau_{\ge 1} Y$.
    By \cref{lem:nilpotence-stability:connective-cover-limit} there is an equivalence $X \cong \tau_{\ge 1} X \cong \tau_{\ge 1} F$.
    It therefore suffices to show that $F$ is connected.
    Since $X \cong X' \times_{Y'} Y$, we know that $A \coloneqq \Fib{X \to X'} \cong \Fib{Y \to Y'}$.
    Similarly, since $F \cong \tau_{\ge 1} X' \times_{\tau_{\ge 1} Y'} \tau_{\ge 1} Y$,
    we know that $B \coloneqq \Fib{F \to \tau_{\ge 1} X'} \cong \Fib{\tau_{\ge 1} Y \to \tau_{\ge 1} Y'}$.
    Thus, we get the following diagrams of long exact sequences of sheaves of pointed sets:
    \begin{center}
        \begin{tikzcd}
            \pi_1(X) \arrow[r] \arrow[d] &\pi_1(X') \arrow[r] \arrow[d] &\pi_0 (A) \arrow[d, equal] \arrow[r] &*\arrow[d] \\
            \pi_1(Y) \arrow[r] & \pi_1(Y') \arrow[r] &\pi_0(A) \arrow[r] &\pi_0(Y),
        \end{tikzcd}
    \end{center}
    and 
    \begin{center}
        \begin{tikzcd}
            \pi_1(F) \arrow[r] \arrow[d] &\pi_1(X') \arrow[r] \arrow[d] &\pi_0 (B) \arrow[d, equal] \arrow[r] &\pi_0(F) \arrow [r]\arrow[d] &* \arrow[d] \\
            \pi_1(Y) \arrow[r] & \pi_1(Y') \arrow[r] &\pi_0(B) \arrow[r] &* \arrow[r] &*.
        \end{tikzcd}
    \end{center}
    From the first diagram, we see that $\pi_0(A) \to \pi_0(Y)$ is the trivial map.
    In particular, $\pi_0(A)$ is the quotient of $\pi_1(Y')$ by $\pi_1(Y)$.
    But by the second diagram, this also true for $\pi_0(B)$.
    Therefore we conclude that $\pi_0(A) \cong \pi_0(B)$.
    But from the first diagram, we see that $\pi_1(X') \to \pi_0(A)$ is surjective.
    In particular, we see in the second diagram that this implies that $\pi_0(F) = *$,
    which is what we wanted to prove.
\end{proof}

\begin{prop} \label{lem:fracture-square:truncated}
    Let $X \in \topos X_*$ be $n$-truncated for some $n \ge 0$ and nilpotent.
    Then the canonical squares
    \begin{equation}
        \label{square:truncated:canonical}
        \begin{tikzcd}
            X \arrow[r] \arrow[d] & \prod_p L_p X \arrow[d] \\
            L_R X \arrow[r] &L_R \prod_p L_p X
        \end{tikzcd}
    \end{equation}
    and 
    \begin{equation}
        \label{square:truncated:connected-canonical}
        \begin{tikzcd}
            X \arrow[r] \arrow[d] & \tau_{\ge 1}\prod_p L_p X \arrow[d] \\
            L_R X \arrow[r] & \tau_{\ge 1} L_R \prod_p L_p X
        \end{tikzcd}
    \end{equation}
    from \cref{lem:fracture-square:canonical-squares}
    are cartesian.
\end{prop}
\begin{proof}
    By \cref{lem:fracture-square:canonical-squares} it suffices to show that the square \eqref{square:truncated:connected-canonical} is cartesian.
    Since $X$ is truncated and nilpotent, there is a sequence of truncated and nilpotent sheaves 
    $X = X_n \xrightarrow{p_n} X_{n-1} \to \dots \to X_0 = *$
    such that $p_k$ fits into a fiber sequence $X_k \to X_{k-1} \to K(A_k, n_k)$
    for some $A_k \in \AbObj{\Disc{\topos X}}$ and $n_k \ge 2$,
    see \cite[Lemma A.15]{mattis2024unstable}.
    We will prove the result by induction on the minimal length of such a sequence,
    the case $n = 0$ is trivial.
    We will write $Y \coloneqq X_{n-1}$ and $K \coloneqq K(A_n, n_n)$,
    i.e.\ there is a fiber sequence $X \to Y \to K$.
    By induction, we know that the canonical square 
    \begin{equation} \label{lem:fracture-square:square:Y}
        \begin{tikzcd}
            Y \arrow[r] \arrow[d] & \tau_{\ge 1} \prod_p L_p Y \arrow[d] \\
            L_R Y \arrow[r] & \tau_{\ge 1} L_R \prod_p L_p Y           
        \end{tikzcd}
    \end{equation}
    is cartesian. Moreover, it follows from \cref{lem:fracture-square:EM-space} that the canonical square 
    \begin{equation} \label{lem:fracture-square:square:EM}
        \begin{tikzcd}
            K \arrow[r] \arrow[d] & \tau_{\ge 1} \prod_p L_p K \arrow[d] \\
            L_R K \arrow[r] &\tau_{\ge 1} L_R \prod_p L_p K           
        \end{tikzcd}
    \end{equation}
    is cartesian.
    Using $X \cong \Fib{Y \to K}$, \cref{lem:nilpotence-stability:p-comp-prod-fiber} now supplies us with an equivalence 
    \begin{equation*}
        \tau_{\ge 1} \prod_p L_p X 
        \cong \tau_{\ge 1} \Fib{\tau_{\ge 1} \prod_p L_p Y \to \tau_{\ge 1} \prod_p L_p K}.
    \end{equation*}
    This then also implies 
    \begin{align*}
        \tau_{\ge 1} L_R \prod_p L_p X 
        &\cong L_R \tau_{\ge 1} \prod_p L_p X \\
        &\cong L_R \tau_{\ge 1} \Fib{\tau_{\ge 1} \prod_p L_p Y \to \tau_{\ge 1} \prod_p L_p K} \\
        &\cong \tau_{\ge 1} \Fib{L_R \tau_{\ge 1} \prod_p L_p Y \to L_R \tau_{\ge 1} \prod_p L_p K} \\
        &\cong \tau_{\ge 1} \Fib{\tau_{\ge 1} L_R \prod_p L_p Y \to \tau_{\ge 1} L_R \prod_p L_p K},
    \end{align*}
    where we used \cref{lem:localization:connected-cover-R-loc} in the first and last equivalence, 
    the above equivalence for the second equivalence,
    and \cref{lem:localization:fiber-cover} in the third equivalence (note that the relevant sheaves 
    are nilpotent by \cref{lem:nilpotence-stability:nilpotent-product}).
    Moreover, \cref{lem:localization:fiber-cover} supplies us with an equivalence 
    \begin{equation*}
        L_R X \cong L_R \Fib{Y \to K} \cong \Fib{L_R Y \to L_R K},
    \end{equation*}
    since $X$ is connected.
    Thus, plugging in these equivalences, we have to show that the following square 
    is cartesian:
    \begin{center}
        \begin{tikzcd}
            \Fib{Y \to K} \arrow[r] \arrow[d] &\tau_{\ge 1} \Fib{\tau_{\ge 1} \prod_p L_p Y \to \tau_{\ge 1} \prod_p L_p K} \arrow[d] \\
            \Fib{L_R Y \to L_R K}\arrow[r] &\tau_{\ge 1} \Fib{\tau_{\ge 1} L_R \prod_p L_p Y \to \tau_{\ge 1} L_R \prod_p L_p K}.
        \end{tikzcd}
    \end{center}
    Since limits commute with limits,
    and since the squares \eqref{lem:fracture-square:square:Y} and \eqref{lem:fracture-square:square:EM}
    are cartesian, we know that the square
    \begin{center}
        \begin{tikzcd}
            \Fib{Y \to K} \arrow[r] \arrow[d] &\Fib{\tau_{\ge 1}\prod_p L_p Y \to \tau_{\ge 1}\prod_p L_p K} \arrow[d] \\
            \Fib{L_R Y \to L_R K}\arrow[r] &\Fib{\tau_{\ge 1}L_R \prod_p L_p Y \to \tau_{\ge 1} L_R \prod_p L_p K}.
        \end{tikzcd}
    \end{center}
    is cartesian.
    Thus, the lemma follows from an application of \cref{lem:fracture-sqaure:connected-cover-pullback-always},
    as $X \cong \Fib{Y \to K}$ and $L_R X \cong \Fib{L_R Y \to L_R K}$ are connected (use \cref{lem:localization:pi-0-eq} for the second claim).
\end{proof}

\begin{thm} \label{lem:fracture-square:main-thm}
    Let $\topos X$ be an $\infty$-topos with a locally finite-dimensional cover $\col{\topos U_i}{i \in I}$.
    Suppose that $X \in \topos X_*$ is a nilpotent sheaf.
    Then the canonical squares
    \begin{equation}
        \label{square:limit:canonical}
        \begin{tikzcd}
            X \arrow[r] \arrow[d] &\prod_p L_p X \arrow[d] \\
            L_R X \arrow[r] &L_R \prod_p L_p X         
        \end{tikzcd}
    \end{equation}
    and 
    \begin{equation}
        \label{square:limit:connected-canonical}
        \begin{tikzcd}
            X \arrow[r] \arrow[d] &\tau_{\ge 1} \prod_p L_p X \arrow[d] \\
            L_R X \arrow[r] &\tau_{\ge 1} L_R \prod_p L_p X         
        \end{tikzcd}
    \end{equation}
    from \cref{lem:fracture-square:canonical-squares}
    are cartesian.
\end{thm}
\begin{proof}
    It suffices to show that the square \eqref{square:limit:connected-canonical} is cartesian, see \cref{lem:fracture-square:canonical-squares}.
    Note that for all $n$, $\tau_{\le n} X$ is nilpotent (\cref{lem:localization:nilpotence-truncation}) and $n$-truncated.
    Therefore, it follows from \cref{lem:fracture-square:truncated} (and the fact that $L_R$ commutes with $\tau_{\ge 1}$, see \cref{lem:localization:connected-cover-R-loc})
    that there are functorial cartesian squares 
    \begin{center}
        \begin{tikzcd}
            \tau_{\le n} X \arrow[r] \arrow[d] &\tau_{\ge 1} \prod_p L_p \tau_{\le n} X \arrow[d] \\
            L_R \tau_{\le n} X \arrow[r] &L_R \tau_{\ge 1} \prod_p L_p \tau_{\le n} X.
        \end{tikzcd}
    \end{center}
    In particular, since limits commute with limits, we get a cartesian square 
    \begin{center}
        \begin{tikzcd}
            \limil{n} \tau_{\le n} X \arrow[r] \arrow[d] &\limil{n} \tau_{\ge 1} \prod_p L_p \tau_{\le n} X \arrow[d] \\
            \limil{n} L_R \tau_{\le n} X \arrow[r] &\limil{n} L_R \tau_{\ge 1} \prod_p L_p \tau_{\le n} X.  
        \end{tikzcd}
    \end{center}
    Note that since $\topos X$ is Postnikov-complete by \cref{lem:htpydim:postnikov-complete},
    the upper left corner is isomorphic to $X$.
    Using \cref{lem:localization:truncation-R-loc}, 
    we compute $\limil{n} L_R \tau_{\le n} X \cong \limil{n} \tau_{\le n} L_R X \cong L_R X$.
    Note that $(\tau_{\ge 1} \prod_p L_p \tau_{\le n} X)_n$ is a locally highly connected tower
    by \cref{lem:htpydim:prod-p-compl-loc-high-con}.
    Thus, \cref{lem:htpydim:R-loc-of-highly-connected}
    gives us an equivalence
    \begin{equation*}
        \limil{n} L_R \tau_{\ge 1} \prod_p L_p \tau_{\le n} X \cong L_R \limil{n} \tau_{\ge 1} \prod_p L_p \tau_{\le n} X.
    \end{equation*}
    Thus, plugging in these equivalences, we arrive at the cartesian square 
    \begin{center}
        \begin{tikzcd}
            X \arrow[r] \arrow[d] &\limil{n} \tau_{\ge 1} \prod_p L_p \tau_{\le n} X \arrow[d] \\
            L_R X \arrow[r] &L_R \limil{n} \tau_{\ge 1} \prod_p L_p \tau_{\le n} X.  
        \end{tikzcd}
    \end{center}
    Since $X$ and $L_R X$ are connected (see \cref{lem:localization:pi-0-eq}), this square factors over the square 
    \begin{center}
        \begin{tikzcd}
            X \arrow[r] \arrow[d] &\tau_{\ge 1} \limil{n} \tau_{\ge 1} \prod_p L_p \tau_{\le n} X \arrow[d] \\
            L_R X \arrow[r] &\tau_{\ge 1} L_R \limil{n} \tau_{\ge 1} \prod_p L_p \tau_{\le n} X.  
        \end{tikzcd}
    \end{center}
    This square is again cartesian, this is an application of \cref{lem:fracture-sqaure:connected-cover-pullback-always}.
    We have equivalences 
    \begin{equation*}
        \tau_{\ge 1} \limil{n} \tau_{\ge 1} \prod_p L_p \tau_{\le n} X
        \cong \tau_{\ge 1} \limil{n} \prod_p L_p \tau_{\le n} X
        \cong \tau_{\ge 1} \prod_p L_p X
    \end{equation*}
    by \cref{lem:nilpotence-stability:connective-cover-limit,lem:htpydim:p-comp-of-post-tower}.
    Using this calculation and the fact that $L_R$ commutes with $\tau_{\ge 1}$ (see \cref{lem:localization:connected-cover-R-loc}),
    we also get an equivalence 
    \begin{equation*}
        \tau_{\ge 1} L_R \limil{n} \tau_{\ge 1} \prod_p L_p \tau_{\le n} X \cong \tau_{\ge 1} L_R \prod_p L_p X.
    \end{equation*}
    Combining these equivalences, we arrive at the cartesian square
    \begin{center}
        \begin{tikzcd}
            X \arrow[r] \arrow[d] &\tau_{\ge 1} \prod_p L_p X \arrow[d] \\
            L_R X \arrow[r] &\tau_{\ge 1} L_R \prod_p L_p X.  
        \end{tikzcd}
    \end{center}
    This proves the theorem.
\end{proof}

\appendix

\section{Examples of \texorpdfstring{$\infty$}{infinity}-Topoi that admit a Locally Finite-Dimensional Cover}
\label{section:examples}
In this section, we provide a list of examples of $\infty$-topoi that admit a locally finite-dimensional cover,
and therefore are $\infty$-topoi where \cref{lem:fracture-square:main-thm} holds.

The main example is the following:
\begin{exmpl}
    Let $\topos X$ be an $\infty$-topos locally of homotopy dimension $\le n$ for some $n \ge 0$.
    If $\topos X$ has enough points, then $\topos X$ has a locally finite-dimensional cover,
    given by the identity $\id{} \colon \topos X \rightleftarrows \topos X \colon \id{}$.
\end{exmpl}

Spcializing the last example, we get:
\begin{exmpl}
    Let $\Cat C$ be a small $\infty$-category.
    Then $\PrShv{\Cat C} \coloneqq \Fun{}(\op{\Cat C}, \An)$ has a locally finite-dimensional cover,
    as it is locally of homotopy dimension $0$, see \cite[Example 7.2.1.9]{highertopoi}.
\end{exmpl}

The other two examples we provide are coming from motivic homotopy theory:
\begin{prop}
    Let $X$ be a quasicompact quasiseperated scheme of finite Krull dimension,
    and write $\smooth{X}$ for the category of quasicompact smooth $X$-schemes.
    Write $\tau$ for either the Zariski or the Nisnevich topology on $\smooth{X}$ (see \cite[Appendix A]{bachmann2020norms} for a definition 
    of the Nisnevich topology),
    and $\topos X \coloneqq \ShvTop{\tau}{\smooth{X}}$ for the $\infty$-topos of $\tau$-sheaves on $\smooth{X}$.
    Then $\topos X$ admits a locally finite-dimensional cover.
\end{prop}
\begin{proof}
    We first prove the case where $\tau$ is the Zariski topology.
    For any smooth $X$-scheme $U$, write $U_{\zar}$ for the small Zariski site of $U$ (i.e.\ the poset of open subsets of $U$,
    equipped with the Zariski topology).
    Write $\ShvTop{\zar}{U_{\zar}}$ for the $\infty$-topos of sheaves on this site.
    By \cite[Proposition 6.3.5.1]{highertopoi} there is a canonical geometric morphism of $\infty$-topoi 
    \begin{equation*}
        g_U^* \colon \topos X \rightleftarrows \topos X_{/U} \colon g_{U,*},
    \end{equation*}
    such that $g_U^*$ has a further left adjoint $g_{U,!}$.
    Note that $\topos X_{/U} \cong \ShvTop{\zar}{(\smooth{X})_{/U}}$,
    where the site $(\smooth{X})_{/U}$ carries the induced topology.
    The canonical inclusion $U_{\zar} \to (\smooth{X})_{/U}$ is a morphism of sites,
    and therefore induces a geometric morphism 
    \begin{equation*}
        h_U^* \colon \topos X_{/U} \rightleftarrows \ShvTop{\zar}{U_{\zar}} \colon h_{U,*}.
    \end{equation*}
    Note that $h_U^*$ has a further left adjoint $h_{U,!}$:
    By the adjoint functor theorem, it suffices to prove that $h_U^*$ commutes with limits.
    We get the following commutative diagram
    \begin{center}
        \begin{tikzcd}
            \topos X_{/U} \arrow[d, hook, "i"] \arrow[r, "h_U^*"] &\ShvTop{\zar}{U_{\zar}} \arrow[d, hook, "j"] \\
            \PrShv{(\smooth{X})_{/U}} \arrow[r, "k"] &\PrShv{U_{\zar}},
        \end{tikzcd}
    \end{center}
    where the bottom morphism $k$ is given by restriction of functors.
    The vertical morphisms commute with limits (as they are right adjoints to the respective sheafification functors),
    and $k$ commutes with limits because limits of presheaves can be calculated on sections.
    It therefore suffices to show that $ki(X)$ satisfies Zariski descent for every $X \to U \in \topos X_{/U}$.
    This is clear, as any Zariski cover in $U_{\zar}$ is also a Zariski cover in $(\smooth{X})_U$.

    Composing these geometric morphisms, we get a geometric morphism 
    \begin{equation*}
        f_U^* \colon \topos X \rightleftarrows \ShvTop{\zar}{U_{\zar}} \colon f_{U,*},
    \end{equation*}
    such that $f_U^*$ has a further left adjoint $f_{U,!}$.

    Note that $\ShvTop{\zar}{U_{\zar}}$ is locally of homotopy dimension $\le \dim{U}$ 
    by \cite[Corollary 7.2.4.17]{highertopoi} (here we use the assumptions on $X$).
    It has enough points, since it is the $\infty$-topos associated to a topological space.

    The collection of geometric morphisms $(f_U^*)_{U \in \smooth{X}}$ is jointly conservative:
    Suppose that $\phi \colon X \to Y$ is a morphism in $\topos X$,
    and suppose that $f_U^*(\phi)$ is an equivalence for every $U \in \smooth{X}$.
    in particular, by taking global sections, we see that the morphism 
    $(f_U^*(X))(U \to U) \to (f_U^*(Y))(U \to U)$ is an equivalence (where $U \to U$ is the terminal object in $U_{\zar}$).
    Since $h_U^*$ is just given by restriction, 
    this is equivalent to the morphism 
    $(g_U^*(X))(U \to U) \to (g_U^*(Y))(U \to U)$.
    But this in turn now is equivalent to the morphism 
    $\phi(U) \colon X(U) \to Y(U)$, as $g_U^*$ is right adjoint to $g_{!,U}$,
    which maps the terminal object $U \to U$ to $U$.
    Since $U$ was arbitrary, we see that $\phi$ is an equivalence.

    This proves that the collection $(f_U^* \colon \topos X \to \ShvTop{\zar}{U_{\zar}})_U$ 
    forms a locally finite-dimensional cover of $\topos X$.

    If $\tau$ is the Nisnevich topology, one argues similarly,
    but replacing the small zariski site $U_{\zar}$ 
    by the small nisnevich site $U_{et}$, consisting of étale $U_i$-schemes,
    equipped with the Nisnevich topology.
    The rest of the proof is the same, we get that $\ShvTop{\nis}{U_{et}}$
    is locally of homotopy dimension $\le \dim{U}$ from \cite[Theorem 3.18]{clausen2021hyperdescent}
    (and it has enough points, given by henselian local schemes, see \cite[Proposition A.3]{bachmann2020norms}).
\end{proof}

\bibliography{bibliography}

\end{document}